\documentclass[11pt]{amsart}

\usepackage{amsmath, amsthm, amssymb, amsfonts, enumerate}
\usepackage{enumitem}
\usepackage[colorlinks=true,linkcolor=blue,urlcolor=blue]{hyperref}
\usepackage{dsfont}
\usepackage{color}
\usepackage{geometry}
\usepackage{todonotes}
\usepackage{enumerate}
  
\newtheorem{theorem}{Theorem}[section]
\newtheorem{remark}[theorem]{Remark}
\newtheorem{assumption}[theorem]{Assumption}
\newtheorem{lemma}[theorem]{Lemma}
\newtheorem{proposition}[theorem]{Proposition}

\newtheorem{definition}{Definition}[]
\newtheorem{example}{Example}[]
\theoremstyle{plain}

\newcommand{\field}[1]{\mathbb{#1}}

\newcommand{\R}{\field{R}}
\newcommand{\N}{\field{N}}

\renewcommand{\P}{\field{P}}

\usepackage[utf8]{inputenc}
\usepackage[english]{babel}
\usepackage{indentfirst}
\usepackage{graphicx}
\usepackage{pdfpages}
\usepackage[title]{appendix}
\usepackage{fancyhdr}

\DeclareMathOperator*{\argmin}{arg\,min}

\title[Submodular Mean Field Games]{Submodular Mean Field Games:\\ Existence and Approximation of Solutions}
\author[Dianetti]{Jodi Dianetti}
\author[Ferrari]{Giorgio Ferrari}
\author[Fischer]{Markus Fischer}
\author[Nendel]{Max Nendel}


\address{J.~Dianetti: Center for Mathematical Economics (IMW), Bielefeld University, Universit\"atsstrasse 25, 33615, Bielefeld, Germany}
\email{\href{mailto:jodi.dianetti@uni-bielefeld.de}{jodi.dianetti@uni-bielefeld.de}}

\address{G.\ Ferrari: Center for Mathematical Economics (IMW), Bielefeld University, Universit\"atsstrasse 25, 33615, Bielefeld, Germany}
\email{\href{giorgio.ferrari@uni-bielefeld.de}{giorgio.ferrari@uni-bielefeld.de}}

\address{M.~Fischer: Department of Mathematics, University of Padova, via Trieste 63, 35121, Padova, Italy}
\email{\href{mailto:fischer@math.unipd.it}{fischer@math.unipd.it}}

\address{M.\ Nendel: Center for Mathematical Economics (IMW), Bielefeld University, Universit\"atsstrasse 25, 33615, Bielefeld, Germany}
\email{\href{mailto:max.nendel@uni-bielefeld.de}{max.nendel@uni-bielefeld.de}}

\date{\today}

\numberwithin{equation}{section}

\begin{document}

\begin{abstract}
We study mean field games with scalar It{\^o}-type dynamics and costs that are \emph{submodular} with respect to a suitable order relation on the state and measure space. The submodularity assumption has a number of interesting consequences. Firstly, it allows us to prove existence of solutions via an application of Tarski's fixed point theorem, covering cases with discontinuous dependence on the measure variable. Secondly, it ensures that the set of solutions enjoys a lattice structure: in particular, there exist a minimal and a maximal solution. Thirdly, it guarantees that those two solutions can be obtained through a simple learning procedure based on the iterations of the best-response-map. The mean field game is first defined over ordinary stochastic controls, then extended to relaxed controls. Our approach allows also to treat a class of submodular mean field games with common noise in which the representative player at equilibrium interacts with the (conditional) mean of its state's distribution.

\end{abstract} 
\maketitle



\section{Introduction}
\label{intro}

In this paper, we study a representative class of mean field games with \emph{submodular} costs. Mean field games (\mbox{MFGs} for short), as introduced by Lasry and Lions \cite{LasryLions07} and, independently, by Huang, Malham{\'e} and Caines \cite{HuangMalhameCaines06}, are limit models for non-cooperative symmetric $N$-player games with mean field interaction as the number of players $N$ tends to infinity; see, for instance, \cite{Cardaliaguet13} and the recent two-volume work \cite{CarmonaDelarue18}.

Submodular games were first introduced by Topkis in \cite{Topkis79} in the context of static non-cooperative $N$-player games. They are characterized by costs of the players that have decreasing differences with respect to a partial order induced by a lattice on the set of strategy vectors. Because the notion of submodularity is related to that of substitute goods in Economics, submodular games have received large attention in the economic literature (see \cite{Amir}, \cite{MilgromRoberts90}, among many others). A systematic treatment of submodular games can be found in \cite{Topkis11}, \cite{Vives01}, and in the survey \cite{Amir05}.

The submodularity assumption has been applied to mean field games by Adlaka and Johari in \cite{AdlakhaJohari13} for a class of discrete time games with infinite horizon discounted costs, by Wi\c{e}cek in \cite{Wiecek17} for a class of finite state mean field games with total reward up to a time of first exit, and by Carmona, Delarue, and Lacker in \cite{CarmonaDelarueLacker17} for mean field games of timing (optimal stopping), in order to study dynamic models of bank runs in a continuous time setting. It is also worth noticing that mean field games considered in recent works adressing the problem of non-uniqueness of solutions enjoy a submodular structure (see e.g.~\cite{BardiFischer18}, \cite{Cecchin&DaiPra&Fischer&Pelino19}, \cite{DelarueFT19}), even if the latter is not exploited therein.


Here, we consider a class of finite horizon mean field games with It{\^o}-type dynamics. More specifically, the evolution of the state of the representative player is described by a one-dimensional It{\^o} stochastic differential equation (SDE) with random (not necessarily Markovian) coefficients and controlled drift. The diffusion coefficient, while independent of state and control, is possibly degenerate. Deterministic dynamics are thus included as a special case. The measure variable, which represents the distribution of the continuum of ``other'' players, only appears in the (random, not Markovian) cost coefficients with running costs split into two parts, one depending on the control, the other on the measure. The measure-dependent costs are assumed to be submodular with respect to first order stochastic dominance on measures and the standard order relation on states (cf.\ Assumption~\ref{ass.submodularity} below).

The submodularity assumption has a number of remarkable consequences. It yields, in particular, an alternative way of establishing the existence of solutions and gives rise to a simple learning procedure. Existence of solutions to the mean field game can be obtained through Banach's fixed point theorem if the time horizon is small (cf.\ \cite{HuangMalhameCaines06}). For arbitrary time horizons, a version of the Brouwer-Schauder fixed point theorem, including generalizations to multi-valued maps, can be used; cf.\ \cite{Cardaliaguet13} and \cite{Lacker15}. Under the submodularity assumption, existence of solutions can instead be deduced from Tarski's fixed point theorem \cite{Tarski55}. This allows us to cover systems with coefficients that are possibly discontinuous in the measure variable. Another notable consequence of the submodularity is that the set of all solutions for a given initial distribution enjoys a lattice structure so that there are a minimal solution and a maximal solution with respect to the order relation. The existence of multiple solutions is in fact quite common in mean field games (see \cite{BardiFischer18, DelarueFT19} and the references therein), and the submodularity assumption is compatible with this non-uniqueness of solutions. Notice that, in particular (yet relevant) cases, we can also prove the existence of MFG solutions when the dynamics of the state process depends on the  measure (see Subsection \ref{meanfielddep}). Furthermore, with a slight modification of the set up, our lattice-theoretical approach allows to deal with a class of MFGs with common noise, in which the representative agent faces a mean field interaction through the conditional mean of its state given the common noise (see Subsection \ref{sec:commonnoise}). This class of MFGs have been recently considered in \cite{DelarueFT19} and \cite{tchuendom}, where the authors address the issue of the uniqueness and selection of equilibria in a linear-quadratic setting.

The problem of how to find solutions to a mean field game in a constructive way has been addressed by Cardaliaguet and Hadikhanloo \cite{CardaliaguetHadikhanloo17}. They analyze a learning procedure, similar to what is known as fictitious play (cf.\ \cite{HofbauerSandholm02} and the references therein), where the representative agent, starting from an arbitrary flow of measures, computes a new flow of measures by updating the average over past measure flows according to the best response to that average. For potential mean field games, the authors establish convergence of this kind of fictitious play. A simpler learning procedure consists in directly iterating the best response map, thus computing a new flow of measures as best response to the previous measure flow. Under the submodularity assumption, we show that this procedure converges to a mean field game solution for appropriately chosen initial measure flows, while it needs not converge for potential or other classes of mean field games.

The rest of this paper is organized as follows. In Subsection~\ref{MFGproblem}, we introduce the controlled system dynamics and costs, together with our standing assumptions, and give the definition of a mean field game, where we take ordinary stochastic open-loop controls as admissible strategies. In Subsection~\ref{latticestructure}, we define the order relation on probability measures which is crucial for the submodularity assumption on the cost coefficients of the game. That assumption is stated and discussed in Subsection~\ref{submodcondition}, while Subsection~\ref{bestresponsemap} deals with properties of the best response map. Subsection~\ref{mainresults} contains our main results, namely Theorem~\ref{theorem.existence} on the existence and lattice structure of MFG solutions and Theorem~\ref{theorem.convergence.learning.procedure} on the convergence of the simple learning procedure. In Section~\ref{section.realxed.mean.field.games}, we extend the analysis of Section~\ref{section.submodular.mean.field.game} to submodular mean field games defined over stochastic relaxed controls. This allows to re-obtain the existence and, especially, the convergence result under more general conditions. Section~\ref{concludingremarks} concludes with comments on the linear-quadratic case, systems with multiplicative and mean field dependent dynamics, and mean field games with common noise. Some auxiliary results on first order stochastic dominance are collected in the Appendix \ref{appendA}.\\

{\bf Notation.} Throughout the rest of this paper, given $x,y\in \mathbb{R}$, we set $x \land y := \min\{x,y\}$ and $x\lor y:=\max \{ x,y \}$. Moreover, given a probability space $(\Omega,\mathcal{F},\mathbb{P})$ and a random variable $X \colon\Omega	\to \mathbb{R}$, we use the (not quite standard) notation $\mathbb{P} \circ X$ for the law of $X$ under $\mathbb{P}$, i.e., we set $\mathbb{P} \circ X [E]:=\mathbb{P}[ X\in E]$ for each Borel set $E$ of $\mathbb{R}.$ Finally, for a given $T\in (0,\infty)$ and a stochastic process $X=(X_t)_{t\in [0,T]}$, with a slight abuse of notation, we denote by $\mathbb{P} \circ X $ the flow of measures associated to $X$; that is, we set $\mathbb{P} \circ X := (\mathbb{P}\circ X_t)_{t\in [0,T]}.$


\section{The submodular mean field game} \label{section.submodular.mean.field.game}

In this section we develop our set up for submodular mean field games. This set up allows us to prove the existence of MFG solutions without using a weak formulation or the notion of relaxed controls. Instead, we combine probabilistic arguments together with a lattice-theoretical approach in order to prove the existence and approximation of MFG solutions. 
 
\subsection{The mean field game problem} \label{MFGproblem}

Let $T>0$ be a fixed time horizon and $W=(W_t)_{t\in [0,T]}$ be a Brownian Motion on a complete filtered probability space $\big(\Omega, \mathcal{F}, (\mathcal F_t)_{t\in [0,T]},  \mathbb{P}\big)$. Let $\xi \in L^2(\Omega,\mathcal F_0,\mathbb P)$ and $(\sigma_t)_{t\in [0,T]} \subset [0,\infty)$ be a progressively measurable square integrable stochastic process. Notice that we allow the volatility process to be zero on a progressively measurable set $E \subset [0,T] \times \Omega$ with positive measure, thus leading to a degenerate dynamics. 
For a set of controls $U\subset \mathbb{R}$, define the the set of admissible controls $\mathcal{A}$ as the set of all square integrable progressively measurable processes $\alpha\colon \Omega \times [0,T] \rightarrow U$. For a measurable function $b\colon \Omega\times [0,T]\times \R\times U\to \R$ and an admissible process $\alpha$,  we consider the controlled SDE (SDE($\alpha$), in short)
\begin{equation}
\label{dynamics} 
dX_t= b(t,X_t, \alpha_t)dt +\sigma_t dW_t,\quad \quad t \in [0,T], \quad X_0=\xi.\\
\end{equation}
With no further reference, thoughout this paper we will assume that for each $(x,a)\in \mathbb{R} \times U$ the process $b(\cdot,\cdot,x,a)$ is progressively measurable and that the usual Lipschitz continuity and growth conditions are satisfied; that is, there exists a constant $C_1>0$ such that for each $(\omega,t,a) \in \Omega \times [0,T] \times U$ we have
\begin{align}\label{assumption.on.b}
   & |b(\omega,t,x,a)-b(\omega,t,y,a)| \leq C_1|x-y|, \quad \forall x,y \in \mathbb{R}, \\
   & |b(\omega,t,x,a)| \leq C_1(1+|x|+|a|^2), \quad \forall x\in \mathbb{R}. \notag
\end{align}
Under the standing assumption, by standard SDE theory, for each $\alpha \in \mathcal{A}$ there exists a unique strong solution $X^\alpha:=(X_t^\alpha)_{t\in [0,T]}$ to the controlled SDE($\alpha$) \eqref{dynamics}.

Let $\mathcal{P}(\mathbb{R})$ denote the space of all probability measures on the Borel $\sigma$-algebra $\mathcal B(\R)$, endowed with the classical ($C_b$-)weak topology,~i.e.\ the topology induced by the weak convergence of probability measures. The costs of the problem are given by three measurable functions 
\begin{align}\label{assumption.costs}
f:&\,\Omega \times [0,T] \times \mathbb{R} \times \mathcal{P}(\mathbb{R})\rightarrow \mathbb{R},  \notag \\
l:&\,\Omega \times [0,T] \times \mathbb{R} \times U\rightarrow \mathbb{R}, \\
g:&\,\Omega \times \mathbb{R} \times \mathcal{P}(\mathbb{R})\rightarrow \mathbb{R}, \notag
\end{align}
such that, for each $(x,\mu,a) \in \mathbb{R} \times \mathcal{P}(\mathbb{R}) \times U$, the processes $f(\cdot,\cdot,x,\mu), \ l(\cdot,\cdot,x,a)$ are progressively measurable and the random variable $g(\cdot,x,\mu)$ is $\mathcal{F}_T$-measurable. We underline that the cost processes $f$ and $g$ are not necessarily Markovian.

For any given and fixed measurable flow $\mu=\left( \mu_t \right)_{t \in [0,T]}$ of probability measures on $\mathcal B(\mathbb{R})$, we
introduce the cost functional
\begin{equation}
\label{cost.functional}
 J(\alpha,\mu):= \mathbb{E} \left[ \int_0^T{\Big[f(t,X_t^{\alpha}, \mu_t)+l(t,X_t^{\alpha}, \alpha_t)\Big] dt} + g(X_T^{\alpha}, \mu_T) \right] , \quad \alpha \in \mathcal{A},
\end{equation}
and consider the optimal control problem $\inf_{\alpha \in \mathcal{A}} J(\alpha, \mu)$.

 We say that $(X^{\mu},\alpha^{\mu})$ is an \emph{optimal pair} for the flow $\mu$ if $-\infty < J(\alpha^{\mu}, \mu) \leq J(\alpha, \mu) $ for each admissible $\alpha\in \mathcal A$ and $X^\mu=X^{\alpha^\mu}$.

\begin{remark}
The subsequent results of this paper remain valid if we consider a geometric dynamics for $X$. Moreover, for suitable choices of the costs, we can also allow for geometric or mean-reverting state processes with dependence on the measure in the dynamics (see subsections \ref{geosec} and \ref{meanfielddep} for more details).
\end{remark}

We make the following standing assumption.

\begin{assumption}\label{ass.main}\
\begin{enumerate} 
\item\label{ass.exist.minima}
For each measurable flow $ \mu $  of probability measures on $\mathcal B(\mathbb R) $, there exists a unique (up to indistinguishability) optimal pair $(X^{\mu},\alpha^{\mu})$.
\item There exists a continuous and strictly increasing function $\psi\colon [0,\infty)\to [0,\infty)$ with $\lim_{s\to \infty}\psi(s)=\infty$ and a constant $M> \psi(0)$ such that
\begin{equation}\label{integrabilityassumption}
	\mathbb E\big[\psi\big(|X_t^{\mu}|\big)\big] \leq M \quad \text{for all measurable flows of probabilities } \mu \text{ and }t \in [0,T].
\end{equation}
\end{enumerate}
\end{assumption}

\begin{remark} 
To underline the flexibility of our set up, Condition (1) in Assumption \ref{ass.main} is stated at an informal level. 
Condition (1) holds, for example, in the case of a linear-convex setting in which $b(t,x,a)=c_t + p_t x + q_t a$, for suitable processes $c$, $p$, $q$, $l(t,\cdot,\cdot)$ is strictly convex and lower semicontinuous, $f(t,\cdot,\mu)$ and $g(\cdot, \mu)$ are lower semicontinuous, and $U$ is convex and compact. More general conditions ensuring existence and uniqueness of an optimal pair in the strong formulation of the control problem can be found in \cite{FuhrmanTessitore04} and in Chapter~II of \cite{Carmona16}, among others. 
\end{remark} 
    
\begin{remark}\label{remark.ass.tightness.minimizers} 
Notice that Condition (2) in Assumption \ref{ass.main} is equivalent to the tightness of the family of laws $\big\{\mathbb{P} \circ X_t^\mu : \mu \text{ is a measurable flow, }\; t\in [0,T]\big\}$ (cf. \cite{chandra}, \cite{leskela} or \cite{nendel}).
The latter is satisfied, for example, if $U$ is compact or if $b$ is bounded in $a$. Alternatively, one can assume that $U$ is closed and that there exist exponents $p'>p\geq 1$ and constants $\kappa,\, K >0$ such that $\mathbb{E}[|\xi|^{p'}] < \infty$ and
\begin{align}\label{ass.alternativegrowth}
&|g(x,\mu)| \leq K(1+|x|^p) ,\\
&\kappa |a|^{p'} - K(1+|x|^p) \leq f(t,x,\mu) + l(t,x,a) \leq K(1+|x|^p+|a|^p), \notag
\end{align} 
for all $ (t,x,\mu,a) \in [0,T] \times \mathbb{R} \times \mathcal{P}(\mathbb{R}) \times U $. Indeed, following the proof of Lemma 5.1 in \cite{Lacker15}, these conditions allow to have an a priori bound on the $p$-moments of the minimizers independent of the measure $\mu$.
\end{remark}
\begin{remark}
Differently from the standard conditions in the literature on mean field games, our existence result (Theorem \ref{theorem.existence}) does not require any continuity of the costs $f$ and $g$ in the measure $\mu$. 
\end{remark}

For each measurable flow $\mu$ of probability measures on $\mathcal B(\R)$, we now define the \emph{best-response} by $R(\mu):=\mathbb{P} \circ X^\mu$, where we set $\mathbb{P} \circ X^\mu := \big( \mathbb{P} \circ X_t^\mu \big)_{t\in [0,T]}$. The map $\mu\mapsto R(\mu)$ is called the \emph{best-response-map}.
\begin{definition}[MFG Solution]
A measurable flow $\mu^*$ of probability measures on $\mathcal B(\R)$ is a mean field game solution if it is a fixed point of the best-response-map $R$; that is, if $R(\mu^*)=\mu^*$.
\end{definition}

\subsection{The lattice structure} \label{latticestructure}

In this section, we endow the space of measurable flows with a suitable lattice structure, which is fundamental for the subsequent analysis. We start by identifying the set of probability measures $\mathcal{P}(\mathbb{R})$ by the set of distribution functions on $\mathbb{R}$, setting $\mu(s):=\mu(-\infty,s]$ for each $s\in \mathbb{R}$ and $\mu\in \mathcal P(\R)$. On $\mathcal{P}(\mathbb{R})$ we then consider the order relation $\leq^{\text{st}}$ given by the \emph{first order stochastic dominance}, i.e.\ we write
\begin{equation}\label{fistorderstochdom}
 \mu \leq^{\text{st}} \nu  \text{ for }\mu, \nu \in \mathcal{P}(\mathbb{R})\text{ if and only if }\mu(s) \geq \nu(s)\text{ for each }s \in \mathbb{R}.
\end{equation}
 The partially ordered set $(\mathcal{P}(\mathbb{R}),\leq^{\text{st}})$ is then endowed with a lattice structure by defining  
\begin{equation}\label{fistorderstochdomminmax}
(\mu \land^{\text{st}} \nu) (s):=\mu(s) \lor \nu (s) \quad \text{and} \quad 
(\mu \lor^{\text{st}} \nu) (s):=\mu(s) \land \nu (s) \quad \text{for each } s\in \mathbb{R}.
\end{equation}
Observe that (see e.g.\ \cite{ShakedShanthikumar07}), for $\mu,\nu \in \mathcal{P}(\mathbb{R})$, we have 
\begin{equation}\label{charact.stoch.dominance}
\mu \leq^{\text{st}} \nu \text{ if and only if }\left\langle \varphi, \mu \right\rangle \leq \left\langle \varphi, \nu \right\rangle
\end{equation}
for each increasing function $\varphi\colon\mathbb{R}\rightarrow \mathbb{R}$ such that $\left\langle \varphi, \mu \right\rangle$ and $\left\langle \varphi, \nu \right\rangle$ are finite,  where $\left\langle \varphi, \mu \right\rangle:=\int_\mathbb{R} \varphi(y)d\mu(y)$.

Recall that by \eqref{integrabilityassumption},
\[
 \mathbb E\big[\psi\big(|X_t^{\mu}|\big)\big] \leq M \quad \text{for all measurable flows $\mu$ and }t \in [0,T].
\]
Then, by Lemma \ref{equivtight}, there exist $\mu^{\rm Min},\, \mu^{\rm Max}\in \mathcal P(\R)$ with
$$\mu^{\rm Min} \leq^{\text{st}} \mathbb{P} \circ X_t^{\mu} \leq^{\text{st}} \mu^{\rm Max} \quad  \text{for all measurable flows $\mu$ and }t \in [0,T].$$
This observation suggests to consider the interval
$$[\mu^{\rm Min},\mu^{\rm Max}]=\big\{ \mu \in \mathcal{P}(\mathbb{R}) \, | \, \mu^{\rm Min}\leq^{\text{st}} \mu \leq^{\text{st}} \mu^{\rm Max} \big\}$$
endowed with the Borel $\sigma$-algebra induced by the weak topology, i.e.\ the topology related to the weak convergence of probability measures. We consider the finite measure $\pi:= \delta_0+dt+\delta_T$ on the Borel $\sigma$-algebra $\mathcal{B}([0,T])$ of the interval $[0,T]$, where $\delta_t$ denotes the Dirac measure at time $t\in [0,T]$. Notice that we include $\delta_0$ into the definition of the measure $\pi$ in order to prescribe the initial law $\mathbb{P} \circ \xi$. We then define the set $L$ of feasible flows of measures as the set of all equivalence classes (w.r.t.\ $\pi$) of measurable flows $(\mu_t)_{t\in [0,T]}$ with $\mu_t\in [\mu^{\rm Min},\mu^{\rm Max}]$ for $\pi$-almost all $t\in (0,T]$ and $\mu_0=\mathbb{P} \circ \xi$. On $L$ we consider the order relation $\leq^{\text{\tiny{$L$}}}$ given by $\mu \leq^{\text{\tiny{$L$}}} \nu$ if and only if $\mu_t \leq^{\text{st}} \nu_t$ for $\pi$-a.a.\ $t\in [0,T]$. This order relation implies that $L$ can be endowed with the lattice structure given by
$$
(\mu \land^{\text{\tiny{$L$}}} \nu)_t:=\mu_t \land^{\text{st}} \nu_t \quad \text{and} \quad 
(\mu \lor^{\text{\tiny{$L$}}} \nu)_t:=\mu_t \lor^{\text{st}} \nu_t \quad \text{for } \pi\text{-a.a. } t  \in [0,T].
$$ 
Notice that $\big(\mathbb{P} \circ X_t^{\alpha} \big)_{t \in [0,T]} \in L$ for every $\alpha\in \mathcal A$. In particular, the best-response-map $R\colon L \rightarrow L$ is well defined. 
\begin{remark}\label{role.p}
We point out that if $\psi(x)=x^2$, then each element of $[\mu^{\rm Min},\mu^{\rm Max}]$ has finite first-order moment, i.e. $ \int_\R |y| d\mu(y) < \infty$ for each $ [\mu^{\rm Min},\mu^{\rm Max}]$. This follows directly from Lemma \ref{remintervalui}.  Notice also that a higher integrability requirement in \eqref{integrabilityassumption} implies the existence and uniform boundedness of higher moments for the elements of $[\mu^{\rm Min},\mu^{\rm Max}]$. More precisely, if $\psi(x)=x^{p'}$ for some $p' \in (1,\infty)$, then $$\sup_{\mu\in [\mu^{\rm Min},\mu^{\rm Max}]}\int_\R |y|^{p} d\mu(y) < \infty\quad \text{for all }p\in (1,p').$$
\end{remark}

We now turn our focus on the main result of this subsection, which is the following lemma.
Its proof follows from the more general Proposition \ref{Append.Completeness}, which is relegated to the Appendix \ref{appendA}. 
\begin{lemma}\label{L.complete}
The lattice $(L,\leq^{\text{\tiny{$L$}}})$ is complete. That is, each subset of $L$ has a least upper bound and a greatest lower bound.
\end{lemma}

\subsection{The submodularity condition} \label{submodcondition}

Our subsequent results rely on the following key assumption.
\begin{assumption}[Submodularity condition]
\label{ass.submodularity} 
For  $\mathbb{P}\otimes dt$ a.a.\ $(\omega,t) \in \Omega \times [0,T]$, the functions $f(t,\cdot, \cdot)$ and $g$  have decreasing differences in $(x,\mu)$; that is, for $\phi \in \{f(t,\cdot, \cdot), g \}$,
\begin{equation*}
 \phi(\bar{x},\bar{\mu})- \phi(x,\bar{\mu}) \leq \phi(\bar{x},\mu)-  \phi(x,\mu ), 
\end{equation*}
for all $ \bar{x},x \in \mathbb{R}$ and $ \bar{\mu},\mu \in \mathcal{P}(\mathbb{R})$ s.t. $\bar{x} \geq x$ and $\bar{\mu} \geq^{\text{st}} \mu$.
\end{assumption}

We list here three examples in which Assumption \ref{ass.submodularity} is satisfied.
\begin{example}
Assumption \ref{ass.submodularity} is always fulfilled for additively separable functions, i.e.\ when $\phi(x,\mu)=\phi_1(x) + \phi_2(\mu)$. 
\end{example} 

\begin{example}[Mean-field interaction of scalar type]\label{example.scalar.type}
Consider a mean-field interaction of scalar type; that is, $\phi(x,\mu)=\gamma(x, 	\left\langle \varphi, \mu \right\rangle )$ for given measurable maps $\gamma:\mathbb{R}^2\rightarrow \mathbb{R}$ and  $\varphi:\mathbb{R}\rightarrow\mathbb{R}$.
If the map $\varphi$ is increasing and the map $\gamma:\mathbb{R}^2 \rightarrow\mathbb{R}$ has decreasing differences in $(x,y) \in \mathbb{R}^2$, then Assumption \ref{ass.submodularity} is satisfied.
Observe that a function $\gamma \in \mathcal{C}^2(\mathbb{R}^2)$ has decreasing differences in $(x,y)$ if and only if $$\frac{\partial^2 \gamma}{ \partial x \partial y}(x,y) \leq 0 \quad \text{for each} \quad (x,y) \in \mathbb{R}^2. $$
\end{example}
\begin{example}[Mean-field interactions of order-1]
Another example is provided by the interactions of order-1, i.e.\ when $\phi$ is of the form
$$
\phi(x,\mu)=\int_\mathbb{R} \gamma(x,y) d\mu(y).
$$
It is easy to check that, thanks to (\ref{charact.stoch.dominance}), Assumption \ref{ass.submodularity} holds when $\gamma$ has decreasing differences in $(x,y)$.
\end{example}
 
A natural and relevant question related to Assumption \ref{ass.submodularity} concerns its link to the so-called \emph{Lasry-Lions monotonity condition}, i.e.\ the condition
\begin{equation}
\label{ass.monotonicity.Lasry.Lion}
    \int_\mathbb{R} (\phi(x, \bar{\mu}) - \phi(x, {\mu}))d(\bar{\mu}-\mu)(x) \geq 0, \quad \forall \, \bar{\mu},\mu \in \mathcal{P}(\mathbb{R}).
\end{equation}
In general, there is no relation between the submodularity condition and \eqref{ass.monotonicity.Lasry.Lion}. However, since Assumption \ref{ass.submodularity} is equivalent to the fact that the map $\phi(\cdot, \bar\mu)-\phi(\cdot, \mu)$ is decreasing for $\mu, \bar\mu\in \mathcal\P(\R)$ with $\bar\mu\geq^{\text{st}}\mu$, Assumption \ref{ass.submodularity} and \eqref{charact.stoch.dominance} imply that
\begin{equation*}
    \int_\mathbb{R} (\phi(x, \bar{\mu}) - \phi(x, {\mu}))d(\bar{\mu}-\mu)(x) \leq 0, \quad \forall \, \bar{\mu},\mu \in \mathcal{P}(\mathbb{R}) \text{ with } \bar\mu\geq^{\text{st}}\mu;
\end{equation*}
the latter, roughly speaking, being sort of an opposite version of the Lasry-Lions monotonicity condition \eqref{ass.monotonicity.Lasry.Lion}.

\begin{remark}
Specific cost functions satisfying Assumption \ref{ass.submodularity} are, for example, 
$$f(t,x,\mu)\equiv 0, \quad l(t,x,a) = \frac{a^2}{2}, \quad g(x,\mu) = \big(x - \mathds{1}_{[0,\infty)}(\langle \text{id}, \mu \rangle)\big)^2,$$
where $\text{id}(y)=y$. Notice that the function $\mu \mapsto g(x,\mu)$ is discontinuous, in contrast to the typical continuity requirement assumed in the literature (see, e.g., \cite{Lacker15}). However, in this specific case, Assumption \ref{ass.main} is only satisfied if the control set $U$ is compact. 
\end{remark}


\subsection{The best-response-map} \label{bestresponsemap}

In the following lemma, we show that the set of admissible trajectories is a lattice.
\begin{lemma}\label{lemma.trajectories.lattice} 
If $\alpha$ and $\bar{\alpha}$ are admissible controls, then there exists an admissible control $\alpha^{\vee}$ such that $X^\alpha \lor X^{\bar{\alpha}}=X^{\alpha^{\vee}}$. Moreover, there exists an admissible control $\alpha^\wedge$ such that $X^\alpha \land X^{\bar{\alpha}}=X^{\alpha^\wedge}$.
\end{lemma}
\begin{proof}
Let $\alpha$ and $\bar{\alpha}$ be admissible controls and define the process $\alpha^{\vee}$ by
$$
\alpha_s^{\vee} := 
\begin{cases}
\alpha_s
& \text{on} \quad \{ X_s^{\alpha} > X_s^{\bar{\alpha}} \} \cup \{ X_s^{\alpha} = X_s^{\bar{\alpha}}, \  b(s,X_s^{\alpha},\alpha_s) \geq b(s,X_s^{\bar{\alpha}},\bar{\alpha}_s) \},\\  
\bar{\alpha}_s 
& \text{on} \quad \{ X_s^{\alpha} < X_s^{\bar{\alpha}} \} \cup \{ X_s^{\alpha} = X_s^{\bar{\alpha}}, \   b(s,X_s^{\alpha},\alpha_s) < b(s,X_s^{\bar{\alpha}},\bar{\alpha}_s) \}.\\ 
\end{cases}
$$
The process $\alpha^{\vee}$ is clearly progressively measurable and square integrable, hence admissible.

We want to show that  $X^\alpha \lor X^{\bar{\alpha}}=X^{\alpha^{\vee}};$  that is,
\begin{equation}\label{SDE.sup}
X_t^\alpha \lor X_t^{\bar{\alpha}} = x_0 + \int_0^t b(s, X_s^\alpha \lor X_s^{\bar{\alpha}} , \alpha_s^{\vee} )ds + \int_0^t \sigma_s dW_s, \quad \forall t \in [0,T], \quad \mathbb{P}\text{-a.s.}
\end{equation}
In order to do so, observe that the process $X^\alpha \lor X^{\bar{\alpha}}$ satisfies, $\mathbb{P}\text{-a.s.}$  for each $t \in [0,T]$, the following integral equation 
\begin{equation}\label{integral.equation.sup}
 X_t^\alpha \lor X_t^{\bar{\alpha}} = x_0 + \int_0^t \sigma_s  dW_s + \bigg( \int_0^t b(s,X_s^{\alpha},\alpha_s) ds \bigg) \lor \bigg(\int_0^t b(s,X_s^{\bar{\alpha}},\bar{\alpha}_s) ds \bigg).
\end{equation}
Furthermore, defining the two processes $A$ and $\bar{A}$ by
$$
A_t := \int_0^t b(s,X_s^{\alpha},\alpha_s) ds \quad \text{and} \quad \bar{A}_t := \int_0^t b(s,X_s^{\bar{\alpha}},\bar{\alpha}_s) ds, 
$$
we see that the process $S$, defined by
$ S_t:=A_t \lor \bar{A}_t$,
is $\mathbb{P}$-a.s.\ absolutely continuous. Hence the time derivative of $S$ exists a.e.\ in $[0,T]$ and, in view of (\ref{integral.equation.sup}), in order to prove (\ref{SDE.sup}) it sufficies to show that $dS_t/dt=b(t, X_t^\alpha \lor X_t^{\bar{\alpha}}, \alpha_t^{\vee})$ for $\mathbb{P} \otimes dt $ a.a.\ $(\omega,t) \in \Omega \times [0,T]$.

Since the processes $A$, $\bar{A}$ and $S$ are $\mathbb{P}$-a.s.\ absolutely continuous, for each $\omega$ in a set of full probability, the paths $A(\omega)$, $\bar{A}(\omega)$ and $S(\omega)$ admit time derivatives in a subset $E(\omega) \subset [0,T]$ with full Lebesgue measure.
We now use a pathwise argument, without stressing the dependence on $\omega \in \Omega$. Take $t \in E$ such that $X_t^{\alpha} > X_t^{\bar{\alpha}}$. 
By continuity, there exists a (random) neighborhood $I_t$ of $t$ in $\mathbb{R}$ such that $X_s^{\alpha} > X_s^{\bar{\alpha}}$ for each $s \in I_t \cap [0,T]$, which, by (\ref{integral.equation.sup}), is true if and only if $A_s > \bar{A}_s$ for each $s \in I_t \cap [0,T]$.
Hence, by definition of $S$, we have
$$ 
\frac{dS_s}{ds} =\frac{dA_s}{ds}= b(s,X_s^{\alpha},{\alpha}_s), \quad \forall \,s \in I_t \cap [0,T] , 
$$
and, in particular, $dS_s/ds = b(s, X_s^\alpha \lor X_s^{\bar{\alpha}}, \alpha_s^{\vee})$ for each $s \in I_t \cap [0,T]$. 

Take now $t \in E$ such that $ X_t^{\alpha} = X_t^{\bar{\alpha}}$ and $b(t,X_t^{\alpha},\alpha_t) \geq b(t,X_t^{\bar{\alpha}},\bar{\alpha}_t)$. From  (\ref{integral.equation.sup}) it follows that $A_t = \bar{A}_t$, which in turn implies that 
$$
\frac{dS_t}{dt}=\lim_{h \to 0} \frac{A_{t+h}\lor \bar{A}_{t+h} - A_t\lor \bar{A}_t}{h} \geq \frac{dA_t}{dt} \lor \frac{d \bar{A}_t}{dt}.
$$
In particular,
\begin{equation} \label{deriv}
	\frac{dA_t}{dt} =b(t,X_t^{\alpha},\alpha_t) \geq b(t,X_t^{\bar{\alpha}},\bar{\alpha}_t)= \frac{d\bar{A}_t}{dt}.
\end{equation} 
If there exists a sequence $\{ h^j \}_{j \in \mathbb{N}}$ converging to $0$ such that $A_{t+h^j} \geq \bar{A}_{t+h^j}$ for each $j\in \mathbb{N}$, then clearly $dS_t/dt = dA_t/dt = b(t,X_t^{\alpha},\alpha_t) = b(t, X_t^{\alpha} \lor X_t^{\bar{\alpha}}, \alpha_t^{\vee}) $, as desired. 
On the other hand, if such a sequence does not exist, then there exists some $\delta >0$ such that  $A_{t+h} \leq \bar{A}_{t+h}$ for each $h\in (-\delta,\delta)$. Recalling (\ref{deriv}), this implies that $d{A}_t/dt \leq dS_t/dt = d\bar{A}_t/dt \leq d{A}_t/dt $, hence we obtain again that $dS_t/dt =  d{A}_t/dt $. 

Altogether, we have proved that for a.a.\ $t \in [0,T]$ with $X_t^{\alpha} > X_t^{\bar{\alpha}}$ or $ X_t^{\alpha} = X_t^{\bar{\alpha}}$ and $b(t,X_t^{\alpha},\alpha_t) \geq b(t,X_t^{\bar{\alpha}},\bar{\alpha}_t)$, we have $dS_t/dt= b(t,X_t^{\alpha},\alpha_t) = b(t, X_t^{\alpha} \lor X_t^{\bar{\alpha}}, \alpha_t^{\vee}) $. 
Analogously, one can prove that $dS_t/dt =b(t,X_t^{\bar{\alpha}},\bar{\alpha}_t)= b(t, X_t^{\alpha} \lor X_t^{\bar{\alpha}}, \alpha_t^{\vee})$ for a.a.\ $t \in [0,T]$ with $X_t^{\alpha} < X_t^{\bar{\alpha}}$ or $X_t^{\alpha} = X_t^{\bar{\alpha}}$ and  $b(t,X_t^{\alpha},\alpha_t) < b(t,X_t^{\bar{\alpha}},\bar{\alpha_t})$. Therefore $dS_t/dt=b(t, X_t^\alpha \lor X_t^{\bar{\alpha}}, \alpha_t^{\vee})$ for $\mathbb{P} \otimes dt $ a.a.\ $(\omega,t) \in \Omega \times [0,T]$, which proves (\ref{SDE.sup}).

The arguments employed above allow to prove that the process $X^\alpha \land X^{\bar{\alpha}}$ satisfies the SDE controlled by $\alpha^\wedge$; i.e.
\begin{equation*}
X_t^\alpha \land X_t^{\bar{\alpha}} = x_0 + \int_0^t b(s, X_s^\alpha \land X_s^{\bar{\alpha}} , \alpha_s^\wedge )ds + \int_0^t \sigma_s dW_s, \quad \forall t \in [0,T], \quad \mathbb{P}\text{-a.s.},
\end{equation*} 
where $\alpha^\wedge$ is defined by 
$$
\alpha^\wedge_s := 
\begin{cases}
\bar{\alpha}_s
& \text{on} \quad \{ X_s^{\alpha} > X_s^{\bar{\alpha}} \} \cup \{ X_s^{\alpha} = X_s^{\bar{\alpha}}, \  b(s,X_s^{\alpha},\alpha_s) \geq b(s,X_s^{\bar{\alpha}},\bar{\alpha}_s) \},\\  
\alpha_s 
& \text{on} \quad \{ X_s^{\alpha} < X_s^{\bar{\alpha}} \} \cup \{ X_s^{\alpha} = X_s^{\bar{\alpha}}, \   b(s,X_s^{\alpha},\alpha_s) < b(s,X_s^{\bar{\alpha}},\bar{\alpha}_s) \}.\\ 
\end{cases} 
$$
The proof of the lemma is therefore completed.
\end{proof}

We now prove the fundamental property of the best-response-map.
\begin{lemma}\label{BRM.increasing}
The best-response-map $R$ is increasing in $(L,\leq^{\text{\tiny{$L$}}})$.	
\end{lemma}
\begin{proof}
Take $\bar{\mu},\mu \in L$ such that $\mu\leq^{\text{\tiny{$L$}}}\bar{\mu}$ and let $(X^{\bar{\mu}},\alpha^{\bar{\mu}})$ and $(X^{\mu},\alpha^{\mu})$ be the optimal pairs related to $\bar{\mu}$ and $\mu$, respectively. 
Define the set 
$$
B := \{ X_s^{\mu} > X_s^{\bar{\mu}} \} \cup \{  X_s^{\mu} = X_s^{\bar{\mu}},  \   b(s,X_s^{\mu},\alpha_s^{\mu}) \geq b(s,X_s^{\bar{\mu}},{\alpha_s}^{\bar{\mu}})\}.
$$
As it is shown in Lemma \ref{lemma.trajectories.lattice}, the process $X^{\mu}\lor X^{\bar{\mu}}$ is the solution to the dynamics (\ref{dynamics}) controlled by $\alpha_t^{\vee}: = \alpha_t^{\mu} \mathds{1}_B(t) + \alpha_t^{\bar{\mu}} \mathds{1}_{B^c}(t)$, and the process $X^{\mu}\land X^{\bar{\mu}}$ is the solution to the dynamics controlled by $\alpha_t^\wedge := \alpha_t^{\mu} \mathds{1}_{B^c}(t) + \alpha_t^{\bar{\mu}} \mathds{1}_{B}(t)$.  

By the admissibility  of $\alpha^{\vee}$ and the optimality of $\alpha^{\bar{\mu}}$ we can write 
\begin{align}\label{J-J}
0 \leq J(\alpha^{\vee},\bar{\mu}) - J(\alpha^{\bar{\mu}},{\bar{\mu}}) = &  \mathbb{E} \bigg[ \int_0^T \Big[f(t,X_t^\mu \lor X_t^{\bar{\mu}}, {\bar{\mu}}_t)  -f(t,X_t^{{\bar{\mu}}}, {\bar{\mu}}_t) \Big] dt  \bigg] \\ \notag
& \quad \quad + \mathbb{E} \bigg[ \int_0^T \Big[ l(t,X_t^{\mu}\lor X_t^{\bar{\mu}}, \alpha_t^{\vee}) -l(t,X_t^{\bar{\mu}}, \alpha_t^{\bar{\mu}}) \Big] dt \bigg] \\ \notag
& \quad \quad \quad \quad + \mathbb{E} \left[ g(X_T^\mu \lor X_T^{\bar{\mu}}, {\bar{\mu}}_T)- g(X_T^{\bar{\mu}}, {\bar{\mu}}_T) \right].
\end{align}
Next, from the definition of $B$ and the trivial identity $1=\mathds{1}_{B}(t) + \mathds{1}_{B^c}(t)$ we find 
\begin{align*}
\mathbb{E} \bigg[ \int_0^T \Big[ f(t,X_t^\mu \lor X_t^{\bar{\mu}}, {\bar{\mu}}_t)  -f(t,X_t^{{\bar{\mu}}}, {\bar{\mu}}_t) \Big] dt  \bigg] &  = \mathbb{E} \bigg[ \int_0^T \mathds{1}_{B}(t) \Big[f(t,X_t^\mu , {\bar{\mu}}_t)  -f(t,X_t^{{\bar{\mu}}}, {\bar{\mu}}_t) \Big] dt  \bigg] \\
& =\mathbb{E} \bigg[ \int_0^T  \Big[ f(t,X_t^\mu , {\bar{\mu}}_t)  -f(t,X_t^\mu \land X_t^{{\bar{\mu}}}, {\bar{\mu}}_t) \Big] dt  \bigg],  
\end{align*}
as well as
$$
\mathbb{E} \left[ g(X_T^\mu \lor X_T^{\bar{\mu}}, {\bar{\mu}}_T)- g(X_T^{\bar{\mu}}, {\bar{\mu}}_T) \right] = \mathbb{E} \left[ g(X_T^\mu , {\bar{\mu}}_T)- g(X_T^\mu \land X_T^{\bar{\mu}}, {\bar{\mu}}_T) \right].
$$
In the same way, by the definition of $\alpha^{\vee}$ and $\alpha^\wedge$, we see that
\begin{align*}
\mathbb{E} \bigg[ \int_0^T \Big[ l(t,X_t^{\mu}\lor X_t^{\bar{\mu}}, \alpha_t^{\vee}) -l(t,X_t^{\bar{\mu}}, \alpha_t^{\bar{\mu}}) \Big] dt \bigg] & =  \mathbb{E} \bigg[ \int_0^T \mathds{1}_{B}(t) \Big[ l(t,X_t^{\mu}, {\alpha}_t^{\mu}) - l(t,X_t^{\bar{\mu}}, \alpha_t^\wedge) \Big] dt \bigg] \\
& = \mathbb{E} \bigg[ \int_0^T \Big[ l(t,X_t^{\mu}, \alpha_t^{\mu} ) -l(t,X_t^{\mu} \land X_t^{\bar{\mu}}, \alpha_t^\wedge) \Big] dt \bigg].
\end{align*}
Now, the latter three equalities allow to rewrite (\ref{J-J})  as 
\begin{align}
\label{J_J.local}
0 \leq J(\alpha^{\vee},\bar{\mu}) - J(\alpha^{\bar{\mu}},{\bar{\mu}})  & =   \mathbb{E} \bigg[ \int_0^T \Big[f(t,X_t^\mu , {\bar{\mu}}_t)  -f(t,X_t^{\mu} \land X_t^{{\bar{\mu}}}, {\bar{\mu}}_t) \Big] dt  \bigg] \\ \notag
& \quad \quad + \mathbb{E} \bigg[ \int_0^T \Big[ l(t,X_t^{\mu}, \alpha_t^{\mu} ) -l(t,X_t^{\mu} \land X_t^{\bar{\mu}}, \alpha_t^\wedge) \Big] dt \bigg] \\ \notag
& \quad \quad \quad \quad + \mathbb{E} \left[ g(X_T^\mu , {\bar{\mu}}_T)- g(X_T^{\mu} \land X_T^{\bar{\mu}}, {\bar{\mu}}_T) \right],
\end{align}
which reads as
\begin{equation}
\label{J_J.switch}
    J(\alpha^{\vee},\bar{\mu}) - J(\alpha^{\bar{\mu}},{\bar{\mu}})=  J(\alpha^{\mu},\bar{\mu}) -  J(\alpha^\wedge,\bar{\mu})
\end{equation}
Finally, exploiting Assumption \ref{ass.submodularity} in the expectations in (\ref{J_J.local}), we deduce that 
\begin{align}
\label{submodularity.of.J}
0 \leq  J(\alpha^{\vee},\bar{\mu}) - J(\alpha^{\bar{\mu}},{\bar{\mu}}) &\leq   \mathbb{E} \bigg[ \int_0^T \Big[ f(t,X_t^\mu , {\mu}_t)  -f(t,X_t^{\mu} \land X_t^{{\bar{\mu}}}, {{\mu}}_t) \Big] dt  \bigg] \\ \notag
& \quad \quad + \mathbb{E} \bigg[ \int_0^T \Big[ l(t,X_t^{\mu}, \alpha_t^{\mu} ) -l(t,X_t^{\mu} \land X_t^{\bar{\mu}}, \alpha_t^\wedge) \Big] dt \bigg] \\ \notag
& \quad \quad \quad \quad + \mathbb{E} \left[ g(X_T^\mu , {{\mu}}_T)- g(X_T^{\mu} \land X_T^{\bar{\mu}}, {{\mu}}_T) \right]\\
&= J(\alpha^{\mu},\mu) -  J(\alpha^\wedge,{\mu}). 
\end{align}
Hence the control $\alpha^\wedge $ is a minimizer for $J(\cdot,\mu)$, and, by uniqueness of the minimizer, we conclude that $X^\mu \land X^{\bar{\mu}}=X^{\mu}$; that is, $X_t^{{\mu}} \leq X_t^{\bar{\mu}}$ for each $t\in [0,T]$ $\mathbb{P}$-a.s., which implies that $R(\mu) \leq^{\text{\tiny{$L$}}} R(\bar{\mu}) $.
\end{proof}
\begin{remark}\label{remark.BRM.increasing}
For later use, we point out that we have actually proved that for $\bar{\mu},\mu \in L$ such that $\mu \leq^{\text{\tiny{$L$}}} \bar{\mu}$ we have that $X_t^{{\mu}} \leq X_t^{\bar{\mu}}$ for each $t\in [0,T]$, $\mathbb{P}$-a.s.
\end{remark}

\subsection{Existence and approximation of MFG solutions} \label{mainresults}

We finally obtain an existence result for the mean field game solutions.

\begin{theorem} \label{theorem.existence}
Under the assumptions \ref{ass.main} and \ref{ass.submodularity},  the set of MFG solutions $(\mathcal{M}, \leq^{\text{\tiny{$L$}}})$ is a nonempty complete lattice: in particular there exist a minimal and a maximal MFG solution.
\end{theorem}

\begin{proof}
Combining Lemma \ref{L.complete} together with Lemma \ref{BRM.increasing}, we have that the best response map $R$ is an increasing map from the complete lattice $(L ,\leq^{\text{\tiny{$L$}}}) $ into itself. The statement then follows from Tarski's fixed point theorem (see Theorem 1 in \cite{Tarski55}).
\end{proof}


Following \cite{Topkis79}, we introduce \emph{learning procedures} $ \{ \underline{\mu}^n \}_{n \in \mathbb{N}}  , \, \{ \overline{\mu}^n \}_{n \in \mathbb{N}} \subset L $ for the mean field game problem as follows:
\begin{itemize}
    \item $\underline{ \mu}^0:=\inf L, \, \overline{ \mu}^0:=\sup L$;
    \item $\underline{\mu}^{n+1}=R(\underline{\mu}^n)$, $\overline{\mu}^{n+1}=R(\overline{\mu}^n)$  for each $n\geq 1$.
\end{itemize}
For simplicity, we make the following assumption. 
\begin{assumption}\label{ass.convergence}\
\begin{enumerate}
\item  The control set $U \subset \mathbb{R}$ is compact and there exists some $p>1$ such that $\mathbb{E}[|\xi|^p]<\infty$.
\item The dynamics of the system given by $b(t,x,a)=c_t + p_t x + q_t a$, where $c,p$ and $q$ are deterministic and continuous in $t$. The volatility $\sigma$ is constant.
\item The cost functions $f,g$ are continuous, the cost function $l$ is convex and lower semicontinuous.
\item $f, \, l$ and $g$ have subpolynomial growth; that is,  there exists a constant $C>0$ such that
$$
|f(t,x,\mu)| +|l(t,x,a)|+ |g(x,\mu)| \leq C (1 + |x|^p ), \ \forall \, (t,x,a,\mu) \in [0,T] \times \mathbb{R} \times U \times [\mu^{\rm Min},\mu^{\rm Max}].
$$
\end{enumerate}
\end{assumption}
 
\begin{remark}\label{remark.continuity.of.the.laws.sol.SDE}
Under Assumption \ref{ass.convergence} it can be easily verified that for each admissible control $\alpha$ the map $t \mapsto \mathbb{P} \circ X_t^{\alpha}$ is continuous in the weak topology. 
\end{remark}

We then have the following convergence result.

\begin{theorem} \label{theorem.convergence.learning.procedure}
Under Assumptions \ref{ass.main}, \ref{ass.submodularity} and \ref{ass.convergence} we have:
\begin{enumerate}
    \item[(i)] The sequence $ \{\underline{\mu}^n \}_{n \in \mathbb{N}}$ is increasing  in $(L ,\leq^{\text{\tiny{$L$}}}) $ and it weakly converges to the minimal  MFG solution, $\pi$-a.e.
    \item[(ii)] The sequence $ \{\overline{\mu}^n \}_{n \in \mathbb{N}}$ is decreasing in $(L ,\leq^{\text{\tiny{$L$}}}) $ and it weakly converges to the maximal MFG solution, $\pi$-a.e.
\end{enumerate}
\end{theorem}

\begin{proof} We only prove the first claim, since the second follows by analogous arguments.

By Lemma \ref{BRM.increasing} the sequence $\{ \underline{\mu}^n\}_{n\in\mathbb{N}}$ is clearly increasing. Moreover, the completeness of the lattice $L$ allows to define $\underline{\mu}^*$ as the least upper bound in the lattice $(L ,\leq^{\text{\tiny{$L$}}}) $ of  $\{ \underline{\mu}^n\}_{n\in\mathbb{N}}$, and, by Remark \ref{append.remark.increasing.sequence.weakly.converges} in Appendix \ref{appendA}, the sequence $\underline{\mu}^n$ converges weakly to $\underline{\mu}^*$ $\pi$-a.e.

Define now, for each $n\geq 1$, the optimal pairs $(X^n, \alpha^n):=(X^{\underline{\mu}^{n-1}},\alpha^{\underline{\mu}^{n-1}})$.
Since the controls $\alpha^n$ take values in the compact set $U$, the processes $ X^n $ are pathwise equicontinuous and equibounded. Moreover, by Remark \ref{remark.BRM.increasing}, the sequence $(X^n)_{n\in \mathbb{N}}$ is increasing. Therefore, by Arzelà-Ascoli's theorem, we can find an adapted  process $X$, such that $X^n$ converges uniformly on $[0,T]$ to $X$, $\mathbb{P}$-a.s. 

We now prove that $\underline{\mu}^*$ is a MFG solution. Since $\underline{\mu}_t^n=R(\underline{\mu}^{n-1})_t=\mathbb{P} \circ X_t^{\underline{\mu}^{n-1}}=\mathbb{P} \circ X_t^n$ and since $X^n$ converges uniformly to $X$ $\mathbb{P}$-a.s.\ and $\underline{\mu}_t^n$ converges weakly to $\underline{\mu}_t^*$ for $\pi$-a.a.\ $t \in [0,T]$, we deduce that $\underline{\mu}_t^*=\mathbb{P} \circ X_t$ for $\pi$-a.a $t\in [0,T]$. Hence, by the continuity of the map $ t \mapsto \mathbb{P} \circ X_t$  in the weak topology (see Remark \ref{remark.continuity.of.the.laws.sol.SDE}), we have $\underline{\mu}_t^*=\mathbb{P} \circ X_t$ for each $t\in [0,T]$.
It remains to find an admissible control $\alpha$ such that $X=X^\alpha$ and $(X,\alpha)$ is the optimal pair for $\underline{\mu}^*$.

In order to do so, thanks to the compactness of $U$, we invoke Banach-Saks' theorem to find a subsequence of indexes $(n_j)_{j\in \mathbb{N} }$ such that the Cesàro means of $(\alpha^{n_j})$ converges pointwise to the process $\alpha$; that is, 
\begin{equation}
\label{cesaro.convergence}
\beta_t^{m}:=\frac{1}{m}\sum_{j=1}^{m} \alpha_t^{n_j} \rightarrow \alpha_t,\text{ as } m\to \infty, \ \mathbb{P}\otimes dt\text{-a.e.}
\end{equation}  
Observe moreover that, by Assumption \ref{ass.convergence}~(2), we have $X^{\beta^m}=\frac{1}{m}\sum_{j=1}^m X^{n_j}$. Hence, because we already know that $X^{n_j}$ converges to $X$ uniformly in $[0,T]$, $\mathbb{P}$-a.s.\ as $n_j\to \infty$, we deduce that   $X^{\beta^m}$ converges uniformly to $X$ $\mathbb{P}$-a.s.\ as $m\to\infty$, and that
$$
X_t = \xi + \int_0^t (c_s + p_s x_s + q_s \alpha_s ) ds + \sigma W_t, \quad \forall \, t \in [0,T], \ \mathbb{P}\text{-a.s.;} 
$$
that is, the process $X$ is the solution to the dynamics controlled by $\alpha$. Furthermore, by the subpolynomial growth of the costs, we have $-\infty < J(\alpha,\underline{\mu}^*)$.

We now prove that the pair $(X,\alpha)$ is optimal for the flow  $\underline{\mu}^*$. Observe that, for each admissible $\beta$ and each $n_j\geq 1 $, by the optimality of the pair $(X^{n_j}, \alpha^{n_j})$ for the flow $\underline{\mu}^{n_j-1}$, we have
$$
J(\alpha^{n_j},\underline{\mu}^{n_j-1}) \leq J(\beta, \underline{\mu}^{n_j-1}) .
$$
Summing over $ j \leq m$, we write
$$
\frac{1}{m}\sum_{j=1}^{m}  \mathbb{E} \left[ \int_0^T \Big[ f(t,X_t^{n_j}, \underline{\mu}_t^{n_j-1}) + l(t, X_t^{n_j}, \alpha_t^{n_j}) \Big] dt + g(X_T^{n_j}, \underline{\mu}_T^{n_j-1}) \right]  \leq \frac{1}{m}\sum_{j=1}^{m}  J(\beta, \underline{\mu}^{n_j-1}),
$$
which, by convexity of $l$, in turn implies that
\begin{align}
\label{last.ineq}
\mathbb{E} \bigg[ \int_0^T  l(t, X_t^{\beta^m}, \beta_t^{m})dt \bigg]+ \frac{1}{m} \sum_{j=1}^{m}  \mathbb{E} \bigg[ \int_0^T  f(t,X_t^{n_j}, \underline{\mu}_t^{n_j-1})   dt   +& g(X_T^{n_j}, \underline{\mu}_T^{n_j-1}) \bigg] \\ \notag
& \leq \frac{1}{m}\sum_{j=1}^{m}  J(\beta, \underline{\mu}^{n_j-1}).
\end{align}
By the convergence of $X^{\beta^m}$ and  $\beta^m$, thanks to the lower semi-continuity and the subpolynomial growth of $l$, we can take limits as $ m \to \infty$ in the first expectation in the latter inequality to find that
\begin{equation}
\label{limit.ell}
\mathbb{E} \bigg[ \int_0^T  l(t, X_t, \alpha_t)dt \bigg] \leq \lim_{m}\mathbb{E} \bigg[ \int_0^T  l(t, X_t^{\beta^m}, \beta_t^{m})dt \bigg].
\end{equation}
Furthermore, by the convergence of $X^n$ and of $\underline{\mu}^n$ and the continuity of the costs $f$ and $g$, we can use the subpolynomial growth of $f$ and $g$ and the boundedness of the sequence $\mu^n$ (cf.\ Remark \ref{role.p}) to deduce that
\begin{equation}
\label{other.limits}
\mathbb{E} \bigg[ \int_0^T  f(t,X_t, \underline{\mu}_t^*)   dt + g(X_T, \underline{\mu}_T^*) \bigg] = \lim_{m} \frac{1}{m} \sum_{j=1}^{m}  \mathbb{E} \bigg[ \int_0^T  f(t,X_t^{n_j}, \underline{\mu}_t^{n_j-1})   dt   + g(X_T^{n_j}, \underline{\mu}_T^{n_j-1}) \bigg]
\end{equation}
and that
\begin{equation}
\label{last.limit}
J(\beta, \underline{\mu}^*) = \lim_{m} \frac{1}{m}\sum_{j=1}^{m}  J(\beta, \underline{\mu}^{n_j-1}).
\end{equation}
Finally, using (\ref{limit.ell}), (\ref{other.limits}) and (\ref{last.limit}) in (\ref{last.ineq}), we conclude that $J(\alpha, \underline{\mu}^*)\leq J(\beta, \underline{\mu}^*)$, which, in turn, proves the optimality of $(X,\alpha)$ for $\underline{\mu}^*$. Hence, $\underline{\mu}^*$ is a MFG solution. 

It only remains to prove the minimality of $\underline{\mu}^*$. Suppose that $\nu^* \in L$ is another MFG solution. By definition, $\inf L = \underline{\mu}^0 \leq^{\text{\tiny{$L$}}} \nu^*$. Since $R$ is increasing, we have $\underline{\mu}^1=R(\underline{\mu}^{0})\leq^{\text{\tiny{$L$}}} R(\nu^*)=\nu^* $ and by induction we conclude that  $\underline{\mu}^n\leq^{\text{\tiny{$L$}}}\nu^*$ for each $n\in \mathbb{N}$. This implies that the same inequality holds for the least upper bound of $\{ \underline{\mu}^n \}_{n\in\mathbb{N}}$; that is, $\underline{\mu}^*\leq^{\text{\tiny{$L$}}} \nu^*$, which completes the proof of the claim. 
\end{proof}

\begin{remark}
\label{remark.minimal.cost.selection} 
In light of Theorem \ref{theorem.convergence.learning.procedure}, a natural question is weather the minimal (resp.\ maximal) MFG solution is associated to the minimal expected cost.
In fact, this relation does not hold in general. Nevertheless, it is easy to see that whenever $f(t,x,\cdot)$ and $g(x,\cdot)$ are increasing (resp.\ decreasing) in $\mu$ for each $(t,x) \in [0,T] \times \mathbb{R}$, the minimal (resp.\ maximal) solution leads to the minimal expected cost and can be approximated via the learning procedure above.
\end{remark}

\begin{remark}
Take  $\mu \in L$ and define the sequence $\mu^0:=\mu$ and $\mu^{n+1}:=R(\mu^{n})$  for $n \in \mathbb{N}$. Following the proof of Theorem \ref{theorem.convergence.learning.procedure} we see that, if $\mu^0 \leq^{L} R(\mu^0)=\mu^1$ (resp. $\mu^0 \geq^{L} R(\mu^0)=\mu^1$), then the sequence $\{ \mu^n \}_{n \in \mathbb{N}}$ is increasing (resp. decreasing) in $(L,\leq^{L})$ and it converges to a MFG equilibrium. In other words, the learning procedure of Theorem \ref{theorem.convergence.learning.procedure} which starts from an arbitrary element converges to a MFG equilibrium whenever the first and the second element of the sequence are comparable.
\end{remark}


\section{Relaxed submodular mean field games}\label{section.realxed.mean.field.games}

In this section we aim at allowing for multiple solutions of the individual optimization problem, and at overcoming the linear-convex setting in the convergence result. This comes with the price of pushing the analysis to a more technical level, by working with a \emph{weak formulation} of the problem and with the so-called \emph{relaxed controls}.

\subsection{The relaxed mean field game} \label{MFGrelaxed}

Let $b,\sigma, f,l,g,U$ be given as in Section \ref{section.submodular.mean.field.game} (see \eqref{assumption.on.b} and \eqref{assumption.costs}), with the additional assumption that $b,f,l,g$ are deterministic and, for simplicity, that $\sigma$ is constant. Let $\mathcal{C}$ denote the set of continuous functions on $[0,T]$. In view of a weak formulation of the problem, the initial value of the dynamics will be described through an initial fixed probability distribution $\nu_0\in \mathcal{P}(\mathbb{R)}$.

Let $\Lambda$ denote the set of \emph{deterministic relaxed controls} on $[0,T] \times	U$; that is, the set of positive measures $\lambda$ on $[0,T] \times U$ such that $\lambda([s,t] \times U)=t-s$ for all $s,t \in [0,T]$ with $s<t$. 
\begin{definition}
\label{def.admissible.relaxed.control}
A 7-tuple $\rho=(\Omega,\mathcal{F},\mathbb{F}, \mathbb{P}, \xi,W,\lambda)$ is said to be an admissible relaxed control if 
\begin{enumerate}
\item $(\Omega,\mathcal{F},\mathbb{F}, \mathbb{P})$ is a filtered probability space satisfying the usual conditions;
\item $\xi$ is an $\mathcal{F}_0$-measurable $\mathbb{R}$-valued random variable (r.v.) such that $\mathbb{P}\circ \xi = \nu_0$;
\item $W=(W_t)_{t \in [0,T]}$ is a standard $(\Omega,\mathcal{F},\mathbb{F}, \mathbb{P})$-Brownian motion;
\item $\lambda$ is a $\Lambda$-valued r.v.\ defined on $\Omega$ such that $\sigma\{ \lambda([0,t]\times E )\,|\,  E \in \mathcal{B}(U) \} \subset \mathcal{F}_t, \ \forall t \in [0,T].$
\end{enumerate}
We denote by $\widetilde{\mathcal{A}}$ the set of admissible relaxed controls.
\end{definition}
As it is shown Lemma 3.2 in \cite{Lacker15}, given $\rho=(\Omega,\mathcal{F},\mathbb{F}, \mathbb{P}, \xi,W,\lambda) \in \widetilde{\mathcal{A}}$, with a slight  abuse of notation, we can define a process $\lambda:\Omega \times [0,T]\to \mathcal{P}(U)$ such that $\lambda(dt,da)=\lambda_t(da)dt$ $\mathbb{P}\otimes dt$-a.e. Through such a  disintegration, we see that the set of admissible controls is naturally included in the set of relaxed controls via the map  $\alpha \mapsto \lambda^\alpha(dt,da):= \delta_{\alpha_t}(da)dt$. 

Furthermore, since $b$ is assumed to satisfy the usual Lipschitz continuity and growth conditions, there exists a unique process $X^\rho:\Omega \times [0,T] \to \mathbb{R}$, solving the system's dynamics equation that now reads as
\begin{equation}
 \label{dynamics.relaxed.controls}
 X_t^\rho = \xi + \int_0^t \int_U b(t,X_t^\rho,a) \lambda_t(da)dt + \sigma W_t, \ t \in [0,T]. 
\end{equation}
Then, for a measurable flow of probability measures $\mu$, we define the cost functional
$$
\widetilde{J}(\rho,\mu):= \mathbb{E}^{\mathbb{P}}\bigg[ \int_0^T\int_U \Big[ f(t,X_t^\rho,\mu_t)+l(t,X_t^\rho,a)  \Big] \lambda_t(da) dt + g(X_T^\rho,\mu_T) \bigg], \ \rho \in \widetilde{\mathcal{A}},
$$
and we say that $\rho \in \widetilde{ \mathcal{A} }$ is an \emph{optimal relaxed control} for the flow of measures $\mu$ if it solves the optimal control problem related to $\mu$; that is, if $ - \infty < \widetilde{J}(\rho,\mu) =\inf \widetilde{J}(\cdot,\mu)$.

We now make the following assumptions, which will be employed in the existence result of Theorem \ref{theorem.relaxed.existence}. 
\begin{assumption}\
\label{assumption.relaxed} 
 \begin{enumerate}
  \item The control space $U$ is compact.
  \item The costs $f(t,\cdot,\mu), \, l(t,\cdot,\cdot)$ and $g(\cdot,\mu)$ are lower semicontinuous in $(x,a)$ for each $(t,\mu) \in [0,T] \times \mathcal{P}(\mathbb{R})$. 
  \item There exist exponents $p'>p\geq 1$ and a constant $K >0$ such that $|\nu_0|^{p'}:=\int_{\mathbb{R}} |y|^{p'} d\nu_0 (y) < \infty$ and such that,  for all $ (t,x,\mu,a) \in [0,T] \times \mathbb{R} \times \mathcal{P}(\mathbb{R}) \times U $, 
  \begin{align*}
  &|g(x,\mu)| \leq K(1+|x|^p + |\mu|^p),\\
 &|f(t,x,\mu)| + |l(t,x,a)| \leq K(1+|x|^p+|\mu|^p),
 \end{align*}
where $|\mu|^p= \int_{\mathbb{R}} |y|^{p} d\mu(y).$
  \item $f$ and $g$ satisfy the Submodularity Assumption \ref{ass.submodularity}. 
 \end{enumerate}
 \end{assumption}

\begin{remark}
Alternatively, as discussed also in Remark \ref{remark.ass.tightness.minimizers}, we can replace (1) in Assumption \ref{assumption.relaxed} by requiring $U$ to be closed and the growth condition \eqref{ass.alternativegrowth} to be satisfied.
\end{remark}

\begin{remark}
\label{remark.compactificatio.method}
Under Assumption \ref{assumption.relaxed}, it is well-known that for each measurable flow $\mu$, $\argmin \widetilde{J}(\cdot, \mu)$ is nonempty. This can be proved using the so-called ``compactification-method'' (see e.g.\ \cite{ElKaroui87} and \cite{HaussmannLepeltier90}, among others). For later use, we now sketch the main argument. Let $(\rho_n)_{ n \in \mathbb{N}}$ be a minimizing sequence for $\widetilde{J}(\cdot, \mu)$, with $\rho_n = (\Omega^n,\mathcal{F}^n, \mathbb{F}^n,\mathbb{P}^n,\xi^n,W^n,\lambda^n)$. Then, since $U$ is compact, thanks to the growth conditions on $b$, the sequence $\mathbb{P}^n \circ (\xi^n,W^n,\lambda^n,X^{\rho_n})$ is tight in $ \mathcal{P}(\mathbb{R} \times \mathcal{C} \times \Lambda \times \mathcal{C})$, so that,  up to a subsequence, $\mathbb{P}^n \circ (\xi^n,W^n, \lambda^n, X^{\rho_n})$  weakly converges to a probability measure $\bar{\mathbb{P}} \in \mathcal{P}(\mathbb{R} \times \mathcal{C} \times \Lambda \times \mathcal{C})$. 
Moreover, through a Skorokhod representation argument, we can find an admissible relaxed control $\rho_*=(\Omega_*,\mathcal{F}_*, \mathbb{F}_*,\mathbb{P}_*,\xi_*,W_*,\lambda_*)$ such that $\bar{\mathbb{P}} = \mathbb{P}_* \circ (\xi_*,W_*, \lambda_*, X^{\rho_*} )$. Finally, the continuity assumptions on the costs together with their polynomial growth, allows to conclude that
$$
\widetilde{J}(\rho_*, \mu) \leq \liminf_n \widetilde{J}(\rho_n, \mu) = \inf \widetilde{J}(\cdot,\mu);
$$
i.e., $\rho_* \in \argmin \widetilde{J}(\cdot,\mu)$. In particular, this argument shows that for any sequence $(\rho_n)_{ n \in \mathbb{N}} \subset \argmin \widetilde{J}(\cdot, \mu)$ we can find  an admissible relaxed control $\rho_*=(\Omega_*,\mathcal{F}_*, \mathbb{F}_*,\mathbb{P}_*,\xi_*,W_*,\lambda_*) \in \argmin \widetilde{J}(\cdot, \mu)$ such that, up to a subsequence, $\mathbb{P}^n \circ X^{\rho_n}$ weakly converges to $\mathbb{P}_* \circ X^{\rho_*}$ in $\mathcal{P}(\mathcal{C})$.  
\end{remark}

The compactness of $U$ and \eqref{assumption.on.b} immediately imply that there exists a constant $M>0$ such that,
     $$\mathbb{E}^\mathbb{P}[|X_t^\rho|^{p'}] \leq M, \quad \forall \, t \in [0,T], \ \rho \in \widetilde{\mathcal{A}}. $$
Hence, Lemma \ref{equivtight} in the Appendix \ref{appendA} allows to find $\mu^{\rm Min},\, \mu^{\rm Max}\in \mathcal P(\R)$ with
$$\mu^{\rm Min} \leq^{\text{st}} \mathbb{P} \circ X_t^{\rho} \leq^{\text{st}} \mu^{\rm Max}, \quad \forall \, t \in [0,T], \ \rho \in \widetilde{\mathcal{A}}.$$
Moreover, as it is shown in Remark \ref{role.p}, we have uniform boundedness of the moments 
\begin{equation}
\label{uniform.bound.for.mu^p}
\sup_{\mu \in [\mu^{\text{Min}},\mu^{\text{Max}}] } |\mu|^q <\infty, \quad \forall \, q < p'. 
\end{equation}

Next, define the set of feasible flows of measures $L$ as the set of all equivalence classes (w.r.t.\ $\pi:=\delta_0 +dt+\delta_T$) of measurable flows $(\mu_t)_{t\in [0,T]}$ with $\mu_t\in [\mu^{\rm Min},\mu^{\rm Max}]$ for $\pi$-almost all $t\in (0,T]$ and $\mu_0=\nu_0$. 
Let $2^L$ be the set of all subset of $L$, and define the best-response-correspondence $\mathcal{R}:L \to 2^L$ by 
\begin{equation}
 \mathcal{R}( \mu) := \big\{ \mathbb{P} \circ X^\rho \, | \, \rho \in \argmin  \widetilde{J}(\cdot,\mu) \big\} \subset L, \quad \mu \in L.
\end{equation} 
We can then give the following definition.
\begin{definition}
The flow of measures $\mu^*$ is a relaxed mean field game solution if $\mu^* \in \mathcal{R}(\mu^*)$.
\end{definition}


\subsection{Existence and approximation of relaxed MFG solutions} 
\label{mainresultsrelaxed}

We now move on to proving the existence and approximation of relaxed mean field game solutions. In order to keep a self-contained but concise analysis, the proofs of the subsequent results will be only sketched whenever their arguments follow along the same lines of those employed in the proofs of Section \ref{section.submodular.mean.field.game}.
\begin{lemma}
\label{lemma.relaxed.bestresponse}
Under Assumption \ref{assumption.relaxed}, the best-response-correspondence satisfies  the following:
 \begin{enumerate}
  \item[(i)] For all $\mu\in L$, we have that  $\inf \mathcal{R}(\mu), \, \sup \mathcal{R}(\mu) \in \mathcal{R}(\mu).$
  \item[(ii)] $\inf \mathcal{R}(\mu)\leq^{L} \inf \mathcal{R}(\overline{\mu})$ and $\sup \mathcal{R}(\mu)\leq^{L} \sup \mathcal{R}(\overline{\mu})$ for all $\mu,\overline{\mu}\in L$ with $\mu\leq^{L} \overline{\mu}$. 
 \end{enumerate}
\end{lemma}

\begin{proof}
We prove the two claims separately.
\vspace{0.15cm}

\emph{Proof of (i)}. Take $\mu \in L$. In order to show that $\inf \mathcal{R}(\mu)\in \mathcal{R}(\mu)$, we recall that, as it shown in the proof of Lemma \ref{Append.Completeness} in the Appendix \ref{appendA}, we can select a sequence of relaxed controls $ (\rho_n)_{n \in \mathbb{N}} \subset \argmin J(\cdot, \mu)$ such that  $\inf \{ \mathbb{P}^n\circ X^{\rho_n} | \, n\in \mathbb{N}\} = \inf \mathcal{R}(\mu)$. Without loss of generality, we can assume that the relaxed controls $\rho_n$ are defined on the same stochastic basis; that is, $\rho^n=(\Omega,\mathcal{F}, \mathbb{F},\mathbb{P},\xi,W,\lambda^n )$ for each $n\in \mathbb{N}$. 

We will now employ an inductive scheme. Let $\rho^1,\,\rho^2$ be the first two elements of the sequence $ (\rho_n)_{n \in \mathbb{N}}$.
As in Lemma \ref{lemma.trajectories.lattice}, we can define two $\Lambda$-valued r.v.'s $\lambda^\vee$ and $\lambda^\wedge$ and two admissible relaxed controls $\rho^\vee=(\Omega,\mathcal{F}, \mathbb{F},\mathbb{P},\xi,W,\lambda^\vee )$ and $\rho^\wedge=(\Omega,\mathcal{F}, \mathbb{F},\mathbb{P},\xi,W,\lambda^\wedge )$ such that $X^{\rho_1} \lor X^{\rho_2} = X^{\rho^\vee}$ and $X^{\rho_1} \land X^{\rho_2} = X^{\rho^\wedge}$. Repeating the same arguments which lead to (\ref{J_J.switch}) in the proof of Lemma \ref{BRM.increasing}, we see that 
$$
0 \leq \widetilde{J}(\rho^\vee, \mu) - \widetilde{J}(\rho_1, \mu) = \widetilde{J}(\rho_2, \mu)-\widetilde{J}(\rho^\wedge,\mu)=0,   
$$
which implies that $\mathbb{P} \circ X^{\rho^\wedge}=\mathbb{P} \circ X^{\rho_1} \land X^{\rho_2}\in \mathcal{R}(\mu)$. Moreover, since $X^{\rho_1} \land X^{\rho_2} = X^{\rho^\wedge}$, we obviously have $ \mathbb{P} \circ X^{\rho^\wedge} \leq^{\text{\tiny{$L$}}} \mathbb{P} \circ X^{\rho^1} \land^{\text{\tiny{$L$}}} \mathbb{P} \circ X^{\rho^2}$. 
Repeating this construction inductively, for each $n\in \mathbb{N}$ we find an admissible relaxed control $\rho^{\wedge n}=(\Omega,\mathcal{F}, \mathbb{F},\mathbb{P},\xi,W,\lambda^{\wedge n} )$ such that $\mathbb{P} \circ X^{\rho^{\wedge n}} \in \mathcal{R}(\mu)$ and $\mathbb{P} \circ X^{\rho^{\wedge n}}  \leq^{\text{\tiny{$L$}}} \mathbb{P} \circ X^{\rho^1} \land^L...\land^L \mathbb{P} \circ X^{\rho^n}$. 
Furthermore, the sequence $\mathbb{P} \circ X^{\rho^{\wedge n}}$ is decreasing in $L$, since for each $n$ we have $ X_t^{\rho^{\wedge n}} = X_t^1 \wedge ... \wedge X_t^n \leq X_t^1 \wedge ... \wedge X_t^{n-1}$ for each $t \in [0,T]$ $\mathbb{P}$-a.s.
Hence,
$$\inf \mathcal{R}(\mu) = \inf \{ \mathbb{P}^n\circ X^{\rho_n} | \, n\in \mathbb{N}\}=\inf \{ \mathbb{P} \circ X^{\rho^{\wedge n }} |\, n \in \mathbb{N} \},$$ which implies that the sequence $\mathbb{P} \circ X^{\rho^{\wedge n}}$ converges weakly to $\inf \mathcal{R}(\mu)$, $\pi$-a.e. Since $(\mathbb{P} \circ X^{\rho^{\wedge n }} )_{n \in \mathbb{N}} \subset \mathcal{R}(\mu)$, by the closure property of $\mathcal{R}(\mu)$ (see Remark \ref{remark.compactificatio.method}), 	we conclude that $\inf \mathcal{R}(\mu) \in \mathcal{R}(\mu)$. 

Analogously, it can be shown that $\sup \mathcal{R}(\mu)\in \mathcal{R}(\mu).$ 
\vspace{0.15cm}

\emph{Proof of (ii)}. Let $\mu,\bar{\mu} \in L$ with $ \mu \leq^{\text{\tiny{$L$}}} \bar{\mu}$ and $\rho,\, \bar{\rho} \in \widetilde{\mathcal{A}}$ with $\rho \in \argmin J(\cdot, \mu)$ and $\bar{\rho} \in \argmin J(\cdot,\bar{\mu})$. With no loss of generality, we can assume that $\rho$ and $\bar{\rho}$ are defined on the same stochastic basis; that is, $\rho=(\Omega,\mathcal{F}, \mathbb{F},\mathbb{P},\xi,W,\lambda )$ and $\rho=(\Omega,\mathcal{F}, \mathbb{F},\mathbb{P},\xi,W,\bar{\lambda} )$. 
As in the proof of claim (i), we can define two $\Lambda$-valued r.v.'s $\lambda^\vee$ and $\lambda^\wedge$ and two admissible relaxed controls $\rho^\vee=(\Omega,\mathcal{F}, \mathbb{F},\mathbb{P},\xi,W,\lambda^\vee )$ and $\rho^\wedge=(\Omega,\mathcal{F}, \mathbb{F},\mathbb{P},\xi,W,\lambda^\wedge )$ such that $X^{\rho} \lor X^{\bar{\rho}} = X^{\rho^\vee}$ and $X^{\rho} \land X^{\bar{\rho}} = X^{\rho^\wedge}$.

Repeating the arguments  which lead to (\ref{submodularity.of.J}) in the proof of Lemma \ref{BRM.increasing}, we exploit the submodularity of the costs and the definitions of $\lambda^\vee$ and $\lambda^\wedge$ to find
\begin{equation*}
\label{submod.J}
 0\leq J(\rho^\vee, \bar{\mu})-J(\bar{\rho}, \bar{\mu})\leq J(\rho^\vee,\mu)-J(\bar{\rho},\mu)= J(\rho,\mu)-J(\rho^\wedge,\mu)\leq 0,
\end{equation*}
where the first and the last inequality hold because of the optimality of $\rho$ and $\bar{\rho}$.  Choosing $\rho$ and $\bar{\rho}$ such that  $\mathbb{P} \circ X^{ \bar{\rho}} = \sup \mathcal{R}(\bar{\mu})$ and 
$\mathbb{P} \circ X^{\rho}=\sup \mathcal{R}(\mu)$ we see that $\sup \mathcal{R}(\mu) \leq^{\text{\tiny{$L$}}} \sup \mathcal{R}(\bar{\mu})$. In the same way, choosing $\rho$ and $\bar{\rho}$ such that  $\mathbb{P} \circ X^{ \bar{\rho}} = \inf \mathcal{R}(\bar{\mu})$ and 
$\mathbb{P} \circ X^{\rho}=\inf \mathcal{R}(\mu)$ we conclude that  $\inf \mathcal{R}(\mu) \leq^{\text{\tiny{$L$}}} \inf \mathcal{R}(\bar{ \mu})$.
\end{proof}

\begin{theorem}
\label{theorem.relaxed.existence}
 Under Assumption \ref{assumption.relaxed}, we have that
\begin{enumerate}
  \item[(i)] The set of mean field game solutions $\mathcal{M}$ is a nonempty lattice and admits a minimal and a maximal element.
  \end{enumerate}
  Assume moreover that the costs $f(t,\cdot,\cdot)$ and $g(\cdot,\cdot)$ are continuous in $(x,\mu)$. Then
  \begin{enumerate}
  \item[(ii)]  For $\underline{\mu}^0:=\inf L$ and $\underline{\mu}^n:= \inf \mathcal{R}(\underline{\mu}^{n-1})$ for $n\in \N$, we have that the learning procedure $( \underline{\mu}^n )_{n \in \mathbb{N}}$ is increasing and it weakly converges to $\inf \mathcal{M}$, $\pi$-a.e.
  \item[(iii)] For $\overline{\mu}^0:=\sup L$ and $\overline{\mu}^n:=\sup \mathcal{R}(\overline{\mu}^{n-1})$ for $n\in \N$, we have that the learning procedure $( \overline{\mu}^n )_{n \in \mathbb{N}}$ is decreasing and it weakly converges to $\sup \mathcal{M}$, $\pi$-a.e.
 \end{enumerate}
\end{theorem}

\begin{proof} Claim (i)  follows from Lemma \ref{lemma.relaxed.bestresponse} and Theorem 4.1 in \cite{Vives90}.  

 We now prove only (ii), since the proof of (iii)  is  similar. By Lemma \ref{lemma.relaxed.bestresponse} the sequence $(\underline\mu^n)_{n\in \N}$ is increasing, hence it weakly converges to its least upper bound $ \underline{\mu}_*$, $\pi$-a.e. For each $n\in \mathbb{N}$, let $\rho^n=(\Omega^n,\mathcal{F}^n, \mathbb{F}^n,\mathbb{P}^n,\xi^n,W^n,\lambda^n)$ be an admissible relaxed control such that $\mathbb{P}^n \circ X^{\rho^n} = \inf \mathcal{R}( \underline{\mu}^{n-1})$. As in Remark \ref{remark.compactificatio.method}, the sequence $(\mathbb{P} \circ (\xi^n,W^n,\lambda^n, X^{\rho^n}))_{n \in \mathbb{N}}$ is tight, so that, up to a subsequence, we can assume that the sequence $\mathbb{P}^n \circ (\xi,W, \lambda^n, X^{\rho^n})$ weakly converges to a probability measure $\bar{\mathbb{P}} \in \mathcal{P}(\mathbb{R} \times \mathcal{C} \times \Lambda \times \mathcal{C})$. Moreover, we can find an admissible relaxed control $\rho_*=(\Omega_*,\mathcal{F}_*, \mathbb{F}_*,\mathbb{P}_*,\xi_*,W_*,\lambda_*)$ such that $\bar{\mathbb{P}} = \mathbb{P}_* \circ (\xi_*,W_*, \lambda_*, X^{\rho_*})$, and this implies that $\mu_*=\mathbb{P}_* \circ X^{\rho_*}$. 

 By the optimality of $\rho^n$ for the flow of measures $\underline{\mu}^{n-1}$, we have
 \begin{equation}
 \label{relaxed.learning.procedure}
 \widetilde{J}(\rho^n,\underline{\mu}^{n-1}) \leq \widetilde{J}(\rho, \underline{\mu}^{n-1}), \ \forall \rho \in \widetilde{\mathcal{A}}.
 \end{equation}
 Now, the continuity of the costs $f,\, l$ and $g$, together with their polynomial growth and the uniform integrability condition \eqref{uniform.bound.for.mu^p}, allow to show the continuity of the functional $\widetilde{J}$ along the sequences $(\rho^n,\underline{\mu}^{n-1})_{n\in \N}$ and $(\rho, \underline{\mu}^{n-1})_{n\in \N}$. This in turn enables us to take limits as $n \to \infty$ in (\ref{relaxed.learning.procedure})  and to deduce that  $\widetilde{J}(\rho_*,\underline{\mu}_*) \leq \widetilde{J}(\rho, \underline{\mu}_*)$ for each $ \rho \in \widetilde{\mathcal{A}}$. Hence, $X^{\rho_*}$ is an optimal trajectory for the flow $\underline{\mu}_*$ and, since $\underline{\mu}_*=\mathbb{P}_* \circ X^{\rho_*}$, we have $\underline{\mu}_* \in \mathcal{R}(\underline{\mu}_*)$; that is, $\underline{\mu}_*$ is a mean field game solution.
 
 It remains to show that $\underline{\mu}_*= \inf \mathcal{M}$. Let $\nu \in \mathcal{M}$. By definition, we have $\underline{{\mu}}^0 = \inf L \leq^{\text{\tiny{$L$}}} \nu$. Since $\inf \mathcal{R}$ is increasing by (ii) in Lemma \ref{lemma.relaxed.bestresponse}, $\underline{\mu}^1 = \inf \mathcal{R}(\underline{\mu}^0) \leq^{\text{\tiny{$L$}}} \inf \mathcal{R}(\nu) \leq^{\text{\tiny{$L$}}} \nu$, where the last inequality follows from $\nu \in \mathcal{R}(\nu)$. By induction, we deduce that $\underline{\mu}^n \leq^{\text{\tiny{$L$}}} \nu$ for each $n \in \mathbb{N}$. Recalling that $\underline{\mu}_* = \sup \{ \underline{\mu}^n | n \in \mathbb{N} \}$, we conclude that $\underline{\mu}_* \leq^{\text{\tiny{$L$}}} \nu$, which completes the proof.
\end{proof}


\section{Concluding remarks and further extensions} \label{concludingremarks}

In the following, we provide some comments on our assumptions and further extensions of the techniques elaborated in the previous sections.

\subsection{On linear-quadratic MFG}

Assumption \ref{ass.submodularity} is fulfilled in the linear-quadratic case
\begin{align*}
& b(t,x,a)= c_t + p_t x + q_t a,\\
& f(t,x,\mu) + l(t,x,a)=\frac{1}{2} n_t a^2 + \frac{1}{2}(m_t x + \widehat{m}_t \langle \text{id}, \mu \rangle)^2,\\ 
&g(x,\mu)=\frac{1}{2}(h_t x + \widehat{h}_t  \langle \text{id}, \mu \rangle)^2,
\end{align*}
where $\text{id}(y)=y$, and for deterministic continuous functions $c,p,q,n,m,\widehat{m},h$ and $\widehat{h}$ such that $\inf_{t \in [0,T]} q_t > 0 $, $\inf_{t \in [0,T]} n_t > 0 $, $n_t \widehat{m}_t \leq 0$ and  $h_t \widehat{h}_t \leq 0$ for each $t \in [0,T]$.

However, the tightness condition (2) in Assumption \ref{ass.main} is not satisfied unless we consider a compact control set $U$. In fact, when $U$ is not compact, there is a counterexample in Section 7 of \cite{Lacker15}, which shows that a mean field game solution may not exist.

Nevertheless, our approach allows to treat non-standard linear-quadratic mean field games, as for example the one considered in Subsection 2.2 in \cite{DelarueFT19}.

\subsection{On a geometric dynamics}\label{geosec}

Our results still hold true if we replace \eqref{dynamics}  with a dynamics of the geometric form
\begin{equation}
\label{geodynamics} 
dX_t= b(t,X_t, \alpha_t) X_t dt +\sigma_t X_t dW_t,\quad \quad t \in [0,T], \quad X_0=\xi,\\
\end{equation}
for some square-integrable positive r.v.\ $\xi$, a bounded drift $b$ and a bounded stochastic process $\sigma$. Indeed, for each square-integrable process $\alpha$ there exists a unique strong solution $X^\alpha$ to the latter SDE, and classical estimates show that there exists a constant $M>0$ such that
\[
\sup_{t \in [0,T]} \mathbb{E} [|X_t^\alpha|^2] \leq M; 
\]
hence, the tightness condition in Assumption \ref{ass.main} is satisfied.
Moreover, the solution to \eqref{geodynamics} can be represented as
$$X^{\alpha}_t= \xi \exp\Big( \int_0^t \big(b(s,X^{\alpha}_s, \alpha_s) - \frac{1}{2}\sigma^2_s \big) ds + \int_0^t \sigma_s dW_s\Big), \quad t \in [0,T],$$
and the mapping $x \mapsto \exp(x)$ is monotone. Hence, since $\xi$ is positive, for any couple of admissible controls $\alpha,$ $\bar{\alpha}$, we have that for each $t\in[0,T]$ $\mathbb{P}$-a.s.
$$X^{\bar{\alpha}}_t \geq X^{\alpha}_t \qquad \text{if and only if} \qquad \int_0^t b(s,X^{\bar{\alpha}}_s, \bar{\alpha}_s)ds \geq \int_0^t b(s,X^{\alpha}_s, \alpha_s) ds.$$
The latter property allows to repeat all the arguments employed in the proof of Lemma \ref{lemma.trajectories.lattice} and (mutatis mutandis) to carry on the analysis that lead to the existence results of Theorems \ref{theorem.existence} and \ref{theorem.relaxed.existence}.

\subsection{On mean field dependent dynamics}\label{meanfielddep}
For a suitable choice of the costs $f$, $g$ and $l$, Theorem \ref{theorem.existence} still holds if we have a ``sufficiently simple'' mean field dependence in the dynamics of the system. For the sake of illustration, we discuss here two example.

Let $U$ be a compact subset of $\mathbb{R}$. For any admissible process $\alpha$ and any measurable flow of probability measures $\mu$, consider a state process given by 
\begin{equation}
\label{geometric.browninan.motion}
dX_t = X_t ( \alpha_t + m(\mu_t)) dt + \sigma X_t dW_t, \quad t \in [0,T], \quad X_0 =\xi, 
\end{equation}
where $\xi$ is a positive square-integrable r.v.\ and $m\colon\mathcal{P}(\mathbb{R}) \to \mathbb{R}$ is a bounded function which is measurable with respect to the Borel $\sigma$-algebra associated to the topology of weak convergence of probability measures. Assume moreover that $m$ is increasing with respect to the first order stochastic dominance.

Notice that, for each measurable flow $\mu$ and for each admissible $\alpha$, the SDE \eqref{geometric.browninan.motion} admits the explicit solution
$$
X_t^{\alpha,\mu} =  E_t(\alpha) M_t(\mu), 
$$
where
$$
E_t(\alpha) := \xi \exp \bigg( \int_0^t \big(\alpha_s - \frac{\sigma^2}{2} \big) ds +  \sigma W_t \bigg) \quad \text{and} \quad M_t(\mu):= 
\exp \bigg( \int_0^t m(\mu_s)ds \bigg).
$$
Since $U$ is compact and $m$ is bounded, we can find a constant $K>0$ which is independent of $\mu$, such that
$$
\sup_{t\in [0,T]} \mathbb{E}[|X_t^{\alpha,\mu}|^2] \leq K. 
$$
The latter implies the tightness condition in Assumption \ref{ass.main}. As in Subsection \ref{latticestructure}, this allows us to define a set $L$ of feasible flows of measures, and to show that $(L, \leq^{\text{\tiny{$L$}}})$ is a complete lattice.

Given $\mu \in L$ and two admissible controls $\alpha$ and $\bar{\alpha}$, as in Lemma \ref{lemma.trajectories.lattice} we can construct $\alpha^\vee$ and $\alpha^\wedge$ such that $X_t^{\alpha,\mu} \vee X_t^{\bar{\alpha},\mu} = X_t^{\alpha^\vee,\mu}$ and $X_t^{\alpha,\mu} \wedge X_t^{\bar{\alpha},\mu} = X_t^{\alpha^\wedge,\mu}$. 
Moreover, due to the particular structure of \eqref{geometric.browninan.motion}, the construction of $\alpha^\vee$ and $\alpha^\wedge$ does not depend on $\mu$.

Consider now cost functions $l(t,x,a)=a^2/2$ and $f(t,x,\mu)= x\psi(\mu)$, for a measureable function $\psi\colon \mathcal{P}(\mathbb{R}) \to \mathbb{R}_{-}$ which is decreasing w.r.t.\ the first order stochastic dominance. With such a choice of the costs, the functional $J$ is strictly convex w.r.t.\ $\alpha$. Hence, for each $\mu \in L $, there exists a unique minimizer $\alpha$ of $J(\cdot, \mu)$ (see, e.g., Theorem 5.2 in \cite{Yong&Zhou99}). We then have the following result.

\begin{lemma}\label{lemma41}
The best-response-map $R\colon L\to L$ is increasing.
\end{lemma}
\begin{proof}
Take $\mu,\, \bar{\mu} \in L$ with $\mu \leq^{\text{\tiny{$L$}}} \bar{\mu}$. Let $\alpha \in \argmin J(\cdot, \mu)$ and $\bar{\alpha} \in \argmin J(\cdot, \bar{\mu})$. 
 
 Similar to Lemma \ref{BRM.increasing}, we see that
 \begin{align*}
 0  \geq J(\alpha^\vee, \mu) &- J(\bar{\alpha}, \mu) = \mathbb{E} \bigg[ \int _0^T \bigg( \frac{(\alpha_t^\vee)^2}{2} - \frac{\bar{\alpha}_t^2}{2} + (X_t^{\alpha^\vee,\mu} - X_t^{\bar{\alpha},\mu} ) \psi(\mu_t) \bigg) dt \bigg] \\
 & = \mathbb{E} \bigg[ \int _0^T \bigg( \frac{(\alpha_t^\vee)^2}{2} - \frac{\bar{\alpha}_t^2}{2} + (E_t(\alpha^\vee) - E_t(\bar{\alpha}) ) M_t(\mu) \psi(\mu_t) \bigg) dt \bigg] \\
 & \geq \mathbb{E} \bigg[ \int _0^T \bigg( \frac{(\alpha_t^\vee)^2}{2} - \frac{\bar{\alpha}_t^2}{2} + (E_t(\alpha^\vee) - E_t(\bar{\alpha}) ) M_t(\bar{\mu}) \psi(\bar{\mu}_t) \bigg) dt \bigg] \\
& = \mathbb{E} \bigg[ \int _0^T \bigg( \frac{(\alpha_t^\vee)^2}{2} - \frac{\bar{\alpha}_t^2}{2} + (X_t^{\alpha^\vee,\bar{\mu}} - X_t^{\bar{\alpha},\bar{\mu}} ) \psi(\bar{\mu}_t) \bigg) dt \bigg]
= J(\alpha^\vee, \bar{\mu}) - J(\bar{\alpha}, \bar{\mu}), 
 \end{align*}
 where we have exploited the particular structure of the dynamics and the fact that $\psi$ is negative and decreasing. Hence $\alpha^\vee \in \argmin J(\cdot, \bar{\mu})$, which, by uniqueness, implies that $\alpha^\vee = \bar{\alpha}$. This in turn implies that $E_t(\alpha^\vee)=E_t(\alpha) \vee E_t(\bar{\alpha})= E_t(\bar{\alpha})$. Hence, $E_t(\alpha) \leq E_t(\bar{\alpha})$ and, by monotonicity of $m$, we find $X_t^{\alpha,\mu} =E_t(\alpha)M_t(\mu) \leq E_t(\bar{\alpha})M_t(\bar{\mu}) = X_t^{\bar{\alpha},\bar{\mu}}$ and $R(\mu)= \mathbb{P} \circ X_t^{\alpha,\mu} \leq^{\text{\tiny{$L$}}} \mathbb{P} \circ X_t^{\bar{\alpha},\bar{\mu}} = R(\bar{\mu})$, which completes the proof. 
 \end{proof}

Thanks to Lemma \ref{lemma41}, we can invoke Tarski's fixed point theorem in order to deduce that the set of mean field game equilibria is a nonempty and complete lattice.

\begin{remark}
Analogous statements still hold if we consider a controlled Ornstein–Uhlenbeck process with mean field term in the dynamics; that is, if the state process is given by
\begin{equation}\label{oumeanfielddep}
dX_t = \big(\kappa X_t+\alpha_t + m(\mu_t)\big)dt + \sigma dW_t, \quad t\in [0,T], \quad X_0 = \xi, \quad \text{with }\kappa\in \R\text{ and } \sigma \geq 0,
\end{equation}
for a measurable bounded increasing function $m:\mathcal{P}(\mathbb{R}) \to \mathbb{R}.$
\end{remark}


\subsection{On a class of MFGs with common noise}
\label{sec:commonnoise}

Our approach allows also to treat a class of submodular mean field games with common noise, in which the representative player interacts with the population through the conditional mean of its state given the common noise. We refer to the recent works \cite{DelarueFT19} and \cite{tchuendom} for a related set up. In the following, we provide the main ingredients of the setting and we show that the set of solutions to the considered class of MFGs with common noise is a nonempty complete lattice.

Let $(W_t)_{t\in [0,T]}$ and $ (B_t)_{t \in [0,T] }$ be two independent Brownian motions on a complete filtered probability space $(\Omega,\mathcal F,(\mathcal F_t)_{t\in [0,T]},\mathbb P)$. Let $\xi\in L^2(\Omega,\mathcal F_0,\mathbb P)$, $\sigma\geq 0$, and $\sigma_0>0$. For each $\alpha\in \mathcal A$ (see the beginning of Subsection \ref{MFGproblem}), consider a dynamics of the system given by
\begin{equation}\label{noisesde}
dX_t = b(t,X_t,\alpha_t) dt + \sigma dW_t + \sigma_0 dB_t, \quad t\in [0,T],  \quad X_0 = \xi, 
\end{equation}
for some bounded measurable function $b$ satisfying the first requirement in \eqref{assumption.on.b}. Here, the Brownian motion $B$ stands for the common noise, while $W$ represents the idiosyncratic noises affecting the state processes in the pre-limit $N$-player game. Notice that, since $b$ is bounded, the solution $X^\alpha$ to the SDE \eqref{noisesde} satisfies ($\mathbb P$-a.s.) the estimate
$$
|X_t^\alpha| \leq |\xi| +t \| b \|_\infty + \sigma |W_t| + \sigma_0 |B_t| =: Y_t\quad \text{for all }t\in [0,T]\text{ and }\alpha\in \mathcal A,
$$
with the process $(Y_t)_{t\in [0,T]}$ belonging to $L^2(\Omega\times[0,T])$.

Let $\mathbb{F}^B=(\mathcal{F}_t^B)$ be the natural filtration generated by $B$ augmented by all $\P$-null sets, and define $L$ to be the set of all real-valued $\mathbb F^B$-progressively measurable processes $\mu=(\mu_t)_{t\in [0,T]}$ such that $|\mu_t| \leq Y_t $ $\mathbb{P}$-a.s., for each $t \in [0,T]$. Then, for any given $\mu\in L$, consider the optimization problem $\inf J(\cdot ,\mu)$ with $J$ as in \eqref{cost.functional} (with appropriately measurable functions $f\colon \Omega\times [0,T]\times \R^2\to \R$, $g\colon \Omega\times \R^2\to \R$ and $l$ as before), suppose that a unique optimal pair $(X^\mu,\alpha^\mu)$ exists, and introduce the map $R\colon L \to L$ defined by
 \[
  R(\mu)_t := \mathbb{E}[X_t^\mu | \mathcal{F}_T^B]\quad \text{for }t\in [0,T].
 \]
 Notice that $R(\mu)$ is $\mathbb F^B$-adapted (see Remark 1 in \cite{tchuendom}) and continuous in $t$, and therefore $\mathbb F^B$-progressively measurable.
 \begin{definition}
 A process $\mu \in L$ is a MFG solution to the MFG with common noise if $$\mu_t= \mathbb{E}[X_t^\mu | \mathcal{F}_T^B] \quad \text{for each }t\in [0,T].$$
 \end{definition}
 
 Consider on $L$ the order relation given by $\mu \leq \bar{\mu}$ if and only if $\mu_t \leq \bar{\mu}_t$ $\mathbb{P}\otimes dt$-a.e. 
 Since $L$ is a bounded subset of the Dedekind complete lattice $L^2(\Omega \times [0,T])$, it is a complete lattice.
 Moreover, as in Remark \ref{remark.BRM.increasing}, for $\bar{\mu},\mu \in L$ with $\mu \leq \bar{\mu}$ we have that $X_t^{{\mu}} \leq X_t^{\bar{\mu}}$ for each $t\in [0,T]$, $\mathbb{P}$-a.s., and hence
 $$
 R(\mu)_t =\mathbb{E}[X_t^\mu| \mathcal{F}_T^B] \leq \mathbb{E}[X_t^{\bar{\mu}}| \mathcal{F}_T^B] = R(\bar{\mu})_t,\quad  \mathbb{P}\otimes dt\text{-a.e.},
 $$
 which implies that $R\colon L \to L$ is increasing. Once more, using Tarski's fixed point theorem, we have proved the following result.
 
 \begin{theorem}\label{excommnoise}
 The set of solutions of the MFG with common noise is a nonempty complete lattice.
 \end{theorem}
 
 \begin{remark}
  Notice that the crucial step in order to obtain Theorem \ref{excommnoise} is the inequality $X_t^{{\mu}} \leq X_t^{\bar{\mu}}$, for each $t\in [0,T]$, whenever $\mu \leq \bar{\mu}$. Following the arguments developed in Subsection \ref{meanfielddep} for MFG without common noise, a similar relation can be established also in the case of mean field dependent dynamics as in \eqref{geometric.browninan.motion} or \eqref{oumeanfielddep} with an additional common noise term $\sigma_0dB_t$. Note that the latter mean-reverting dynamics is exactly the one considered in \cite{DelarueFT19} and \cite{tchuendom}.
 \end{remark}
 

\appendix

\section{Some results on first order stochastic dominance}\label{appendA}

In this section, we derive some technical results concerning the first order stochastic dominance introduced in Subsection \ref{latticestructure}. As in Subsection \ref{latticestructure}, we identify the set of probability measures $\mathcal{P}(\mathbb{R})$ by the set of distribution functions on $\mathbb{R}$, setting $\mu(s):=\mu(-\infty,s]$ for each $s\in \mathbb{R}$ and $\mu\in \mathcal P(\R)$. On $\mathcal{P}(\mathbb{R})$ we then consider the lattice ordering of first order stochastic dominance given by \eqref{fistorderstochdom} and \eqref{fistorderstochdomminmax}.

\begin{remark}\label{minimaandmaxima}\
\begin{enumerate}
 \item[a)] Notice that by identifying $\mu$ by its distribution function, $\mathcal P(\R)$ coincides with the set of all nondecreasing right continuous functions $F\colon \R\to [0,1]$ with $\lim_{s\to -\infty}F(s)=0$ and $\lim_{s\to \infty} F(s)=1$. Moreover, we would like to recall that the weak topology is metrizable and that the weak convergence coincides with the pointwise convergence of distribution functions at every continuity point, i.e. $\mu_n\to \mu$ if and only if
 \[
 \mu_n(s)\to \mu(s)\quad\text{as }n\to \infty\quad \text{for every continuity point }s\in \R\text{ of }\mu.
 \]
 Therefore, the weak convergence behaves well with the pointwise lattice operations $\vee^{\rm st}$ and $\wedge^{\rm st}$. In particular, the maps $(\mu,\nu)\mapsto \mu \vee^{\rm st} \nu $ and $(\mu,\nu)\mapsto \mu \wedge^{\rm st} \nu $ are continuous $\mathcal P(\R)\times \mathcal P(\R)\to \mathcal P(\R)$.
 \item[b)] Recall that a nondecreasing function $\R\to \R$ is right continuous if and only if it is upper semicontinuous (usc). Hence, for a sequence $(\mu^n)_{n\in \N}\in \mathcal P(\R)$ which is bounded above, the supremum $\sup_{n\in \N}\mu^n$ is exactly the pointwise infimum of the distribution functions $( \mu^n )_{n \in \mathbb{N}}$.
 \item[c)] For a nondecreasing function $F\colon \R\to \R$, we define its usc-envelope $F^*\colon \R\to \R$ by
\[
 F^*(s):=\inf_{\delta >0}F(s+\delta)\quad \text{for all }s\in \R.
\]
Notice that $F(s)\leq F^*(s)\leq F(s+\varepsilon)$ for all $s\in \R$ and $\varepsilon >0$. Intuitively speaking, $F^*$ is the right continuous version of $F$. That is, $F^*$ differs from $F$ only at discontinuity points of $F$. For a sequence $(\mu^n)_{n\in \N}\in \mathcal P(\R)$ which is bounded below, the infimum $\inf_{n\in \N}\mu^n$ is then given by the usc-envelope of the pointwise supremum of the distribution functions $(\mu^n)_{n\in \N}$. That is, one has to modify the pointwise supremum at all its discontinuity points in order to be right continuous.
\item[d)] Combining the previous remarks, leads to the following insight: If $(\mu^n)_{n\in \N}\subset \mathcal P(\R)$ is a bounded and nondecreasing or nonincreasing sequence, then $(\mu^n)_{n\in \N}$ converges weakly to its supremum or infimum, respectively.
\end{enumerate}
\end{remark}

\begin{lemma}\label{equivtight}
 Let $K\subset \mathcal P(\R)$ and $\psi\colon [0,\infty)\to [0,\infty)$ be continuous and strictly increasing with $\psi(s)\to \infty$ as $s\to \infty$ and
 \[
  \sup_{\mu\in M}\int_\R \psi(|x|) d\mu(x)<\infty.
 \]
 Then, there exist $\mu^{\rm Min},\mu^{\rm Max}\in \mathcal P(\R)$ with $\mu^{\rm Min}\leq^{\rm st} \mu\leq^{\rm st}\mu^{\rm Max}$ for all $\mu\in K$.
\end{lemma}

\begin{proof}
 We extend $\psi$ to $(-\infty,0)$ by $\psi(s):=\psi(0)$ for $s<0$. Moreover, let $C\geq \psi(0)$ with
 \[
 \sup_{\mu\in K}\int_\R \psi(|x|) d\mu(x) \leq C.
 \]
 Then, we define $\mu^{\rm Min},\mu^{\rm Max}\colon \R\to [0,1]$ by
 \begin{equation}\label{muminmax}
  \mu^{\rm Min}(s):=\frac{C}{\psi(-s)}\wedge 1 \quad \text{and} \quad \mu^{\rm Max}(s):=\bigg(1-\frac{C}{\psi(s)}\bigg)\vee 0
 \end{equation}
 for all $s\in \R$. Notice that $\mu^{\rm Min}(s)= 1$ for $s\geq 0$ and $\mu^{\rm Max}=0$ for $s\leq 0$ since $\psi(0)\leq C$. Since $\psi(s)\to \infty$ as $s\to \infty$, it follows that $\lim_{s\to -\infty}\mu^{\rm Min}(s)=0$ and $\lim_{s\to \infty} \mu^{\rm Max}(s)=1$. Moreover, $\mu^{\rm Min}$ and $\mu^{\rm Min}$ are nondecreasing and (right) continuous, which shows that $\mu^{\rm Min},\mu^{\rm Max}\in \mathcal P(\R)$. Now, let $\mu\in K$. Then, recalling that $\psi$ is nondecreasing, one has
 \[
  1-\mu(s)\leq \frac{1}{\psi(s)}\int_s^\infty\psi(|x|) d\mu(x)\leq \frac{1}{\psi(s)}\int_\R \psi(|x|) d\mu(x)\leq \frac{C}{\psi(s)}= 1-\mu^{\rm Max}(s)
 \]
 for all $s\in \R$ with $\psi(s)>C$. Since $\mu^{\rm Max}(s)=0$ for all $s\in \R$ with $\psi(s)\leq C$, it follows that $\mu\leq \mu^{\rm Max}$. On the other hand,
  \[
  \mu(s)\leq \frac{1}{\psi(-s)}\int_{-\infty}^s\psi(|x|) d\mu(x)\leq \frac{1}{\psi(-s)}\int_\R\psi(|x|) d\mu(x)\leq \frac{C}{\psi(-s)}=\mu^{\rm Min}(s)
 \]
 for all $s\in \R$ with $\psi(-s)>C$. Since $\mu^{\rm Min}(s)=1$ for all $s\in \R$ with $\psi(-s)\leq C$, it follows that $\mu\geq \mu^{\rm Min}$.
\end{proof}

\begin{lemma}\label{remintervalui}
 Let $K\subset \mathcal P(\R)$ and $\psi\colon [0,\infty)\to [0,\infty)$ be continuous and strictly increasing with $\psi(s)\to \infty$ as $s\to \infty$ and
 \[
  \sup_{\mu\in M}\int_\R \psi(|x|) d\mu(x)<\infty.
 \]
 Further, let $\mu^{\rm Min}$ and $\mu^{\rm Max}$ be given by \eqref{muminmax} and $0\leq \alpha<1$. Then, the map $ x\mapsto \psi(|x|)^\alpha$ is u.i for $[\mu^{\rm Min},\mu^{\rm Max}]$, i.e.
 \[
  \sup_{\mu\in [\mu^{\rm Min},\mu^{\rm Max}]}\int_\R 1_{(M,\infty)}(|x|)\cdot \psi(|x|)^\alpha d\mu(x)\to 0\quad \text{as } M\to \infty.
 \]
\end{lemma}

\begin{proof}
 Let $\beta \in (\alpha,1)$. Then, by \eqref{muminmax},
 \begin{equation}\label{lemA3}
  \psi(s)=\frac{C}{1-\mu^{\rm Max}(s)}\quad \text{for }s\geq \psi^{-1}(C)\quad \text{and}\quad \psi(-s)=\frac{C}{\mu^{\rm Min}(s)}\quad \text{for }s\leq -\psi^{-1}(C)
 \end{equation}
Recall $\psi^{-1}(C)=\big(\mu^{\rm Max}\big)^{-1}(0)=\big(\mu^{\rm Max}\big)^{-1}(1)$. This together with \eqref{lemA3} implies that
 \[
  \int_0^\infty \psi(s)^\beta d\mu^{\rm Max}(s)=\int_{\psi^{-1}(0)}^\infty\bigg(\frac{C}{1-\mu^{\rm Max}(s)}\bigg)^\beta d\mu^{\rm Max}(s)=\int_0^1 \bigg(\frac{C}{1-u}\bigg)^\beta du<\infty
 \]
 and
 \[
  \int_{-\infty}^0 \psi(-s)^\beta d\mu^{\rm Min}(s)=\int_{-\infty}^{-\psi^{-1}(C)} \bigg(\frac{C}{\mu^{\rm Min}(s)}\bigg)^\beta d\mu^{\rm Min}(s)=\int_0^1 \bigg(\frac{C}{u}\bigg)^\beta du<\infty,
 \]
 where, in both equalities, we used the transformation lemma. It follows that
 \[
  \sup_{\mu\in [\mu^{\rm Min},\mu^{\rm Max}]}\int_\R \psi(|x|)^\beta d\mu(x)\leq \int_0^\infty \psi(s)^\beta d\mu^{\rm Max}(s)+\int_{-\infty}^0 \psi(-s)^\beta d\mu^{\rm Min}(s).
 \]
 By the De La Vall\'{e}e-Poussin Lemma, it follows that $|x|\mapsto \psi(|x|)^\alpha$ is u.i. for $[\mu^{\rm Min},\mu^{\rm Max}]$. In particular, if $\psi(s)\geq s^p$ for some $p\in (0,\infty)$, then, $x\mapsto |x|^q$ is u.i. for $[\mu^{\rm Min},\mu^{\rm Max}]$ for all $q\in (0,p)$.
 
\end{proof}


We now turn our focus on measureable flows of probability measures. The following proposition is the starting point in order to apply Tarski's fixed point theorem in the proof of the existence of mean field game solutions. We start by building up the setup. Let $\underline\mu,\overline\mu\in \mathcal P(\R)$ with $\underline\mu\leq^{\text{st}} \overline\mu$ and $(S,\mathcal S,\pi)$ be a finite measurable space. We denote by $\mathcal B$ the Borel $\sigma$-algebra on $\mathcal P(\R)$ generated by the weak topology. We denote the lattice of all equivalence classes of $\mathcal S$-$\mathcal B$-measurable functions $S\to \mathcal [\underline \mu,\overline \mu]$ by $L=L^0(S,\pi;[\underline \mu,\overline \mu])$. An arbitrary element $\mu$ of $L$ will be denoted in the form $\mu=(\mu_t)_{t\in S}$. On $L$ we consider the order relation $\leq^{\text{\tiny{$L$}}}$ given by $\mu \leq^{\text{\tiny{$L$}}} \nu$ if and only if $\mu_t \leq^{\text{st}} \nu_t$ for $\pi$-a.a.\ $t\in S$. The following proposition can be found in a more general form in \cite{nendel}. However, for the sake of a self-contained exposition, we provide a short proof below.

\begin{proposition}\label{Append.Completeness}
 The lattice $L$ is complete.
\end{proposition}

\begin{proof}
 Let $M\subset L$ be a nonempty subset of $L$. Then, for every countable set $\Psi\subset M$, we denote by $\mu^\Psi:=\sup_{\mu\in \Psi}f\in L$. Notice that the $\mathcal S$-$\mathcal B$-measurability of $\mu^\Psi$ follows from Remark \ref{minimaandmaxima}. Let
 \[
  c:=\sup\bigg\{\int_S \int_\R\arctan(x) d\mu_t^\Psi(x) d\pi(t)\, \bigg|\, \Psi\subset M \text{ countable}\bigg\}.
 \]
 Moreover, the map $t\mapsto \int_\R\arctan(x)\, d\mu_t$ is measurable for every $\mu\in L$ since $\arctan\in C_b(\R)$ induces a continuous (w.r.t. to the weak topology) linear functional $\mathcal P(\R)\to \R$. By definition of the constant $c$, there exists a sequence $(\Psi^n)_{n\in \N}$ of countable subsets of $M$ with
 \[
  \int_S \int_\R\arctan(x) d\mu_t^{\Psi^n}(x) d\pi(t)\to c\quad \text{as }n\to \infty.
 \]
 Let $\Psi^*:=\bigcup_{n\in \N}\Psi^n$ and $\mu^*:=\mu^{\Psi^*}$. We now show that $\mu^*\geq \mu$ $\pi$-a.s. for all $\mu\in M$. In order to see this, fix some $\mu\in M$ and let $\Psi':=\Psi^*\cup\{\mu\}$. Then, it follows that
 \[
  c=\int_S \int_\R\arctan(x) d\mu_t^*(x) d\pi(t)\leq \int_S \int_\R\arctan(x) d\mu_t^{\Psi'}(x) d\pi(t)\leq c. 
 \]
 Since $\arctan$ is strictly increasing it follows that $\mu^{\Psi'}=\mu^*$, i.e. $\mu\leq \mu^*$. Moreover, for any upper bound $\mu\in L$ of $M$ it is easily seen that $\mu\geq \mu^*$. Altogether, we have shown that $\mu^*=\sup M$. In a similar way, one shows that $M$ has an infimum.
\end{proof}

\begin{remark}\label{append.remark.increasing.sequence.weakly.converges}
Let $M\subset L$ be nonempty. Then, we say that $M$ is \textit{directed upwards} or \textit{directed downwards} if for all $\mu,\nu\in M$ there exists some $\eta\in M$ with $\mu\vee \nu\leq \eta$ or $\mu\wedge \nu \geq \eta$, respectively.
\begin{enumerate}
 \item[a)] The proof of the previous theorem shows that if $M$ is directed upwards, then there exists a nondecreasing sequence $(\mu^n)_{n\in \N}\subset M$ with $\mu^n\to \sup M$ weakly $\pi$-a.e.~as $n\to \infty$. The analogous statement holds for the infimum if $M$ is directed downwards. In particular, if $( \mu^n )_{n \in \mathbb{N}}$ is a nondecreasing or nonincreasing sequence in $L$, then it converges weakly $\pi$-a.e.~to its least upper bound or greatest lower bound, respectively.
 \item[b)] Assume that $S$ is a singleton with $\pi(S)>0$. Then, the previous remark implies the following: For any nonempty set $K\subset \mathcal P(\R)$ that is bounded above and directed upwards, its supremum exists and can be weakly approximated by a monotone sequence. An analogous statement holds for the infimum if the set $K$ is bounded below and directed downwards.
\end{enumerate}
\end{remark}

\smallskip
\section*{Acknowledgements}
Financial support through the German Research Foundation via CRC 1283 is gratefully acknowledged.


\end{document}